\documentclass{amsart}
\usepackage[T1]{fontenc}
\usepackage[utf8]{inputenc}
\usepackage{amsmath,amsfonts,amssymb,amsthm}
\usepackage{mathtools}
\usepackage{tensor}
\usepackage{pgf,tikz}
\usetikzlibrary{matrix,arrows,calc,positioning}
\usepackage{tikz-cd}
\usepackage{shuffle}
\usepackage{url}
\usepackage[svgnames]{xcolor}
\numberwithin{equation}{section}
\numberwithin{table}{section}

\usepackage{caption} 
\usepackage{enumitem}

\usepackage[
	colorlinks, 
	citecolor=DarkGreen,
	linkcolor=DarkBlue,
	pdftitle={The motivic Selmer scheme of the thrice-punctured line},
	pdfauthor={Martin Lüdtke},
]{hyperref}
\usepackage{cleveref}

\newtheorem{thm}{Theorem}
\newtheorem{prop}[thm]{Proposition}
\newtheorem{lemma}[thm]{Lemma}
\newtheorem{cor}[thm]{Corollary}

\theoremstyle{remark}
\newtheorem{rem}[thm]{Remark}
\newtheorem{remarks}[thm]{Remarks}
\newtheorem{example}[thm]{Example}

\theoremstyle{definition}
\newtheorem{defn}[thm]{Definition}

\AddToHook{env/defn/begin}{\crefalias{thm}{defn}}
\AddToHook{env/prop/begin}{\crefalias{thm}{prop}}
\AddToHook{env/lemma/begin}{\crefalias{thm}{lemma}}
\AddToHook{env/conjecture/begin}{\crefalias{thm}{conjecture}}
\AddToHook{env/cor/begin}{\crefalias{thm}{cor}}
\AddToHook{env/rem/begin}{\crefalias{thm}{rem}}
\AddToHook{env/remarks/begin}{\crefalias{thm}{remarks}}
\AddToHook{env/example/begin}{\crefalias{thm}{example}}

\numberwithin{thm}{section}

\newcommand{\one}{\mathbf{1}}
\newcommand{\bA}{\mathbb A}

\newcommand{\rC}{\mathrm C}
\newcommand{\bC}{\mathbb C}
\newcommand{\rd}{\mathrm d}
\newcommand{\rf}{\mathrm{f}}
\newcommand{\bF}{\mathbb F}
\newcommand{\fg}{\mathfrak g}
\newcommand{\bG}{\mathbb G}
\newcommand{\cT}{\mathcal T}
\newcommand{\rH}{\mathrm H}
\newcommand{\cH}{\mathcal H}

\newcommand{\bN}{\mathbb N}
\newcommand{\fn}{\mathfrak n}
\newcommand{\MT}{\mathrm{MT}}
\newcommand{\bP}{\mathbb P}
\newcommand{\fp}{\mathfrak p}

\newcommand{\bQ}{\mathbb Q}
\newcommand{\bZ}{\mathbb Z}
\newcommand{\cO}{\mathcal{O}}

\newcommand{\fu}{\mathfrak{u}}
\newcommand{\rZ}{\mathrm{Z}}
\newcommand{\fz}{\mathfrak{z}}

\newcommand{\op}{\mathrm{op}}
\newcommand{\ProFinVect}{\mathrm{ProFinVect}}

\newcommand{\llangle}{\langle\!\langle}
\newcommand{\rrangle}{\rangle\!\rangle}

\newcommand{\gp}{{\mathrm{gp}}}
\newcommand{\real}{{\mathrm{real}}}
\newcommand{\mot}{{\mathrm{mot}}}
\newcommand{\inv}{{\mathrm{inv}}}
\newcommand{\Tors}{{\mathrm{Tors}}}
\newcommand{\Ind}{{\mathrm{Ind}}}
\newcommand{\Pro}{{\mathrm{Pro}}}
\newcommand{\Proj}{{\mathrm{Proj}}}

\newcommand{\et}{\operatorname{\acute{e}t}}
\newcommand{\Gal}{\mathrm{Gal}}

\newcommand{\Sh}{\mathrm{Sh}}
\newcommand{\Vect}{\mathrm{Vect}}
\newcommand{\Betti}{\mathrm{Betti}}
\newcommand{\MTR}{\mathrm{MTR}}
\newcommand{\lto}{\longrightarrow}

\DeclareMathOperator{\ab}{ab}
\DeclareMathOperator{\met}{met}
\DeclareMathOperator{\ad}{ad}

\DeclareMathOperator{\Aut}{Aut}
\DeclareMathOperator{\can}{can}

\DeclareMathOperator{\cris}{cris}
\DeclareMathOperator{\Der}{Der}
\DeclareMathOperator{\Ext}{Ext}

\DeclareMathOperator{\gr}{gr}

\DeclareMathOperator{\dR}{dR}
\DeclareMathOperator{\pr}{pr}
\DeclareMathOperator{\id}{id}

\DeclareMathOperator{\res}{res}
\DeclareMathOperator{\loc}{loc}
\DeclareMathOperator{\Li}{Li}
\DeclareMathOperator{\Lie}{Lie}
\DeclareMathOperator{\Log}{Log}

\DeclareMathOperator{\MF}{MF}

\DeclareMathOperator{\adm}{adm}
\DeclareMathOperator{\Hom}{Hom}

\DeclareMathOperator{\PL}{PL}
\DeclareMathOperator{\Rep}{Rep}
\DeclareMathOperator{\Sel}{Sel}
\DeclareMathOperator{\Spec}{Spec}

\newcommand{\pathtorsor}[3][]{\tensor*[_{#2}]{\Pi}{_{#3}^{#1}}}
\newcommand{\pth}[3][]{\tensor*[_{#2}]{p}{_{#3}^{#1}}}

\mathchardef\mhyphen="2D %

\usepackage[
	backend=biber,
	style=alphabetic,
	minalphanames=3,
	maxnames=99,
	maxalphanames=4,
	giveninits=true,
	isbn=false,
	doi=true,
]{biblatex}
\title{The motivic Selmer scheme of the thrice-punctured line}

\author{Martin Lüdtke}
\address{Martin Lüdtke, Rijksuniversiteit Groningen and Carl von Ossietzky Universität Oldenburg
}
\email{martin.luedtke@uol.de}

\begin{document}

\begin{abstract}
	Let $X = \bP^1 \smallsetminus \{0,1,\infty\}$ be the thrice-punctured line over a ring of $S$-integers $\cO_{K,S}$ in a number field~$K$. For any quotient $\pi_1^{\mot}(X,0) \twoheadrightarrow \Pi$ of Deligne--Goncharov's motivic fundamental group there is an associated Selmer scheme which parametrises $\Pi$-torsors with a mixed Tate motive structure. We give several descriptions of the motivic Selmer scheme which make it amenable to computations, using $\bG_m$-equivariant cocycles of algebraic groups, Lie algebras, and complete Hopf algebras. We prove that the Selmer scheme is isomorphic to an affine space $\bA^N_{\bQ}$, and we construct coordinates realising this isomorphism. This is a key ingredient for making the motivic Chabauty--Kim method explicit in a general setting, without restrictions on the base field or the choice of fundamental group quotient.
\end{abstract}

\maketitle

\tableofcontents
\thispagestyle{empty}

\section{Introduction}
\label{sec: introduction}

Let $K$ be a number field and let $\cO_{K,S}$ be the ring of $S$-integers for a finite set~$S$ of primes of~$K$. Let $X/\Spec(\cO_{K,S})$ be a regular $S$-integral model of a smooth hyperbolic curve over~$K$. The Chabauty--Kim method provides a framework for studying the set $X(\cO_{K,S})$ of $S$-integral points using the arithmetic of non-abelian fundamental groups. The key object in this method is the \emph{Selmer scheme}. In the original formulation by Kim \cite{kim:motivic,kim:albanese}, one chooses an auxiliary prime~$p$ and a base point $b \in X(\cO_{K,S})$, and defines the Selmer scheme as an affine $\bQ_p$-scheme whose points parametrise $G_K$-equivariant torsors under the $\bQ_p$-prounipotent étale fundamental group. The idea can be summarised as follows. Every $S$-integral point~$\alpha \in X(\cO_{K,S})$ gives rise to a point of the Selmer scheme via the \emph{path torsor}. Using a similar construction for local points, one obtains a commutative diagram
\[
\begin{tikzcd}
	X(\cO_{K,S}) \rar[hook] \dar["j"] & \prod_{\fp \mid p} X(\cO_{\fp}) \dar["\prod_{\fp \mid p} j_{\fp}"] \\
	\{\text{global Selmer scheme}\} \rar["\loc_p"] & \prod_{\fp \mid p} \{\text{local $\fp$-adic Selmer scheme}\}
\end{tikzcd}
\]
in which the vertical \emph{Kummer maps} are given by the path torsor construction. In order for a tuple of local points $(z_{\fp})_{\fp} \in \prod_{\fp \mid p} X(\cO_{\fp})$ to come from a global $S$-integral point, it is necessary that its image under the local Kummer map lies in the image of the global Selmer scheme under the localisation map. This produces strong constraints on the location of $X(\cO_{K,S})$ inside the local points. In some cases it is possible to make this explicit and derive a set of Coleman analytic functions on $\prod_{\fp \mid p} X(\cO_{\fp})$ whose vanishing locus contains or is even equal to the set of $S$-integral points. In practice, one usually works with finite-dimensional quotients of the fundamental group, such as quotients along the descending central series, in which case the localisation map becomes an algebraic map of finite type affine $\bQ_p$-schemes.

In this paper we restrict attention to the global side. In the case $X = \bP^1 \smallsetminus \{0,1,\infty\}$, the existence of the category of mixed Tate motives over~$K$ and the construction of the \emph{motivic fundamental group} $\pi_1^{\mot}(X,0)$ by Deligne and Goncharov \cite{deligne-goncharov} make it possible to define a \emph{motivic Selmer scheme} over~$\bQ$ and a motivic Kummer map. The purpose of this paper is to give a systematic description of the motivic Selmer scheme and, in particular, to make its structure sufficiently explicit for computations. We show that for any finite-dimensional quotient $\pi_1^{\mot}(X,0) \twoheadrightarrow \Pi$, the associated Selmer scheme $\Sel_{S,\Pi}^{\mot}(X)$ is isomorphic to $\bA^N$ for some~$N \in \bZ_{\geq 0}$, and compute the dimension for various natural choices of $\Pi$. Moreover, we present a general construction of coordinate functions which make the isomorphism $\Sel_{S,\Pi}^{\mot}(X) \cong \bA^N$ explicit. This lays the groundwork for the Chabauty--Kim calculations contained in subsequent work with Li and Kim (over number fields, using the polylogarithmic quotient of the fundamental group) and with Corwin and Dan-Cohen (over $\bQ$, using the full fundamental group) \cite{LL:PolylogNF,KLL:Cyclotomic,BeyondPolylog}. 

We begin by giving a simple definition of the motivic Selmer scheme as a moduli space of $\Pi$-torsors in the category $\MT(\cO_{K,S},\bQ)$ of $\bQ$-linear mixed Tate motives over~$\cO_{K,S}$. We prove its independence of the base point in a more general setup of fundamental groupoids in Tannakian categories. We show that the motivic Selmer scheme provides a $\bQ$-form of the étale Selmer scheme over~$\bQ_p$, generalising a result from \cite{BKL:chabauty-kim-sc} to the number field setting. As the definition of the motivic Selmer scheme as a moduli space of torsors is rather inexplicit, we then give various other more concrete descriptions: using algebraic group cohomology for the Tannaka group $G_S^{\MT}$ of $\MT(\cO_{K,S},\bQ)$; via $\bG_m$-equivariant algebraic cocycles $U_S^{\MT} \to \Pi^{\omega}$ on the unipotent radical $U_S^{\MT} \subseteq G_S^{\MT}$; via graded Lie algebra cocycles $\Lie(U_S^{\MT}) \to \Lie(\Pi^{\omega})$; and finally via graded Hopf algebra cocycles $\cH(U_S^{\MT}) \to \cH(\Pi^{\omega})$. We also discuss the motivic Kummer map in terms of Selmer scheme coordinates. We have aimed to keep the exposition largely self-contained so that it may serve as a useful reference for future work on motivic Chabauty--Kim theory.  

Much of the theory developed in this paper has precursors in the existing Chabauty--Kim literature. For fundamental group quotients~$\pi_1^{\mot}(X,0) \twoheadrightarrow \Pi$ on which the unipotent Galois group $U_S^{\MT}$ acts trivially, Proposition~5.2.1 of \cite{DCW:mixedtate1} identifies the motivic Selmer scheme with the space of $\bG_m$-equivariant homomorphisms $U_S^{\MT} \to \Pi^{\omega}$. This was used in \cite[§3]{CDC:polylog1} to construct a set of coordinates on the Selmer scheme when $\Pi$ is the \emph{polylogarithmic quotient}, which allowed the construction of explicit Kim functions in the case $K = \bQ$ and $\#S = 1$. Brown \cite{brown:integral_points} studied the motivic Selmer scheme for general fundamental group quotients, including those with non-trivial $U_S^{\MT}$-action. He identifies the Selmer scheme with the space of $\bG_m$-equivariant \emph{cocycles} $U_S^{\MT} \to \Pi^{\omega}$ and proves representability by an affine scheme. 
Our main contribution is to make this more precise by showing representability by an \emph{affine space} $\bA^N$, and constructing coordinates which make this isomorphism explicit. This construction is a crucial ingredient in the follow-up work \cite{BeyondPolylog}, where the use of non-polylogarithmic quotients leads, among other results, to the first proved instances of Kim's Conjecture for the thrice-punctured line in the case $K = \bQ$ and $\#S = 2$.

\subsection*{Acknowledgements}

I would like to thank Alex Betts, David Corwin, Ishai Dan-Cohen, Minhyong Kim, Xiang Li, and Steffen Müller for their encouragement and interest in the topics of this paper. 
The author was supported by an NWO Grant, project number VI.Vidi.192.106, a guest postdoc fellowship at the Max Planck Institute for Mathematics in Bonn, and a Minerva Fellowship of the Minerva Stiftung Gesellschaft für die Forschung mbH.

\subsection*{AI statement}

The manuscript was written entirely by the author. ChatGPT-6 was used only for proofreading.

\section{The motivic Selmer scheme as a moduli space of torsors}
\label{sec:selmer-scheme-moduli-space}

We start by giving a definition of the motivic Selmer scheme as a moduli space parametrising torsors under the motivic fundamental group.

\subsection{Mixed Tate motives}
\label{sec:mixed_tate_motives}

Let $K$ be a number field with ring of integers $\cO_K$. Let $S$ be a finite set of primes of $K$ and let $\cO_{K,S} = \cO(\Spec(\cO_K) \smallsetminus S)$ be the ring of $S$-integers.
Denote by
\[ \MT(\cO_{K,S}, \bQ) \]
the $\bQ$-linear Tannakian category of mixed Tate motives over $K$ which are unramified outside $S$. This category was constructed by Deligne--Goncharov~\cite{deligne-goncharov}. It contains a distinguished object~$\bQ(1)$ whose tensor powers $\bQ(i) \coloneqq \bQ(1)^{\otimes i}$, $i \in \bZ$, are called the \emph{Tate motives}. They form the simple objects of $\MT(\cO_{K,S}, \bQ)$. Every mixed Tate motive $M \in \MT(\cO_{K,S},\bQ)$ carries a canonical increasing weight filtration $W_i M$ indexed by the even integers and characterised by the fact that the graded piece $\gr_{2i}^W(M)$ is a direct sum of copies of $\bQ(-i)$ for all $i \in \bZ$. 

The category of mixed Tate motives is equipped with various realisation functors to other Tannakian categories:
\begin{enumerate}
	\item For each complex embedding $\sigma\colon K \hookrightarrow \bC$ there is a Betti realisation functor to finite-dimensional $\bQ$-vector spaces:
	\[ \real_{\Betti,\sigma}\colon \MT(\cO_{K,S}, \bQ) \to \bQ\mhyphen\Vect_{\rf} \]
	\item For each prime~$p$ there is a $p$-adic étale realisation functor to the category of continuous finite-dimensional $\bQ_p$-linear representations of the absolute Galois group $G_K = \Gal(\overline{K}/K)$:
	\[ \real_{\et,p}\colon \MT(\cO_{K,S}, \bQ) \to \Rep_{\bQ_p}(G_K). \]
	
	\item There is a de Rham realisation functor to the category of finite-dimensional $K$-vector spaces:
	\[ \real_{\dR}\colon \MT(\cO_{K,S}, \bQ) \to K\mhyphen\Vect_{\rf}. \]
	
	\item For each prime $v \not\in S$, there is a filtered $\varphi$ realisation functor to the category of admissible filtered $\varphi$-modules over~$K_v$:
	\[ \real_{F\varphi,v}\colon \MT(\cO_{K,S}, \bQ) \to \MF^{\varphi,\adm}_{K_v}. \]
\end{enumerate}

The different realisation functors are related via various comparison isomorphisms but we will not need them in this paper.

\subsection{Algebraic geometry in a Tannakian category}
\label{sec:alg_geom_in_tannakian_cat}

Let $F$ be a field and let $\cT = (\cT, \otimes, \one)$ be an $F$-linear Tannakian category. The language of ``algebraic geometry in a Tannakian category'' \cite[§5]{deligne:droite-projective} provides an elegant way of equipping $F$-algebras, and dually affine $F$-schemes, with extra structure. By definition, a \emph{ring in $\cT$} is an ind-object $A$ of $\cT$ equipped with a unit $\one \to A$ and multiplication $A \otimes A \to A$ making it a commutative monoid object of the ind-category of $\cT$. Any $F$-algebra $R$ can be viewed as a ring in~$\cT$ by taking sums of the unit object. An \emph{affine $F$-scheme in~$\cT$} is formally of the form $\Spec(A)$ with~$A$ a ring in~$\cT$, and the morphisms $\Spec(B) \to \Spec(A)$ are of the form $\Spec(f)$ with $f\colon A \to B$ a homomorphism of rings in~$\cT$.

\begin{example}
	\label{ex: k-vect category}
	When $\cT$ is the category of finite-dimensional $F$-vector spaces, then an affine scheme in $\cT = F\mhyphen\Vect_{\rf}$ is simply an affine $F$-scheme in the usual sense. Indeed, an ind-object $A$ of $F\mhyphen\Vect_{\rf}$ is just an $F$-vector space (not necessarily finite-dimensional), and a unit map $F \to A$ and multiplication map $A \otimes_F A \to A$ satisfying the axioms of a commutative monoid object give~$A$ the structure of an $F$-algebra. Reversing arrows, this makes $\Spec(A)$ an affine $F$-scheme in the usual sense.
\end{example}

\begin{example}
	\label{ex: representation category}
	Let $\cT = \Rep(G)$ be the $F$-linear Tannakian category of representations of a pro-algebraic group~$G$ over~$F$. An affine $F$-scheme in~$\Rep(G)$ is an affine $F$-scheme~$X$ in the usual sense equipped with a $G$-action $G \times X \to X$ over~$F$. Any affine $F$-scheme can be viewed as an affine $F$-scheme in $\Rep(G)$ by endowing it with the trivial $G$-action.
\end{example}

Various notions from algebraic geometry, such as faithfully flat morphisms, fibre products, affine group schemes and torsors can be defined intrinsically in a Tannakian category; see \cite[§5]{deligne:droite-projective} for details.

\subsection{The motivic fundamental group of the thrice-punctured line}
\label{motivic-fundamental-group}

Let $S$ be a finite set of primes in a number field~$K$. Let $X = \bP^1_{\cO_{K,S}} \smallsetminus \{0,1,\infty\}$ be the thrice-punctured line over the ring of $S$-integers. The (unipotent) \emph{motivic fundamental group} of Deligne and Goncharov \cite{deligne-goncharov} is a pro-unipotent group
\[ \pi_1^{\mot}(X,0) \] 
in $\MT(\cO_{K,S}, \bQ)$, in the sense of §\ref{sec:alg_geom_in_tannakian_cat}. That is, the ring of functions $\cO(\pi_1^{\mot}(X,0))$ has the structure of an ind-mixed Tate motive. When writing~$0$ for the base point we really mean the tangential base point $\vec{1}_0$ of~$X$ at~$0$, given by the tangent vector $\partial /\partial t$ in terms of the coordinate $t$ on~$\bA^1$. Upon taking realisations, $\pi_1^{\mot}(X,0)$ specialises to the various other incarnations of the fundamental group:
\begin{enumerate}
	\item Taking the Betti realisation with respect to a complex embedding $\sigma\colon K \hookrightarrow \bC$ yields the (unipotent) Betti fundamental group:
	\[ \real_{\Betti,\sigma}(\pi_1^{\mot}(X,0)) = \pi_1^{\mathrm{Betti}}(X_{\sigma}(\bC), 0). \]
	The latter agrees with the Malcev completion of the topological fundamental group $\pi_1^{\mathrm{top}}(X_{\sigma}(\bC),0)$. Here, $X_{\sigma}(\bC)$ denotes the space of $\bC$-valued points of $X$ with respect to the embedding $\sigma\colon K \hookrightarrow \bC$. We refer to \cite[§15]{deligne:droite-projective} for how to define the fundamental group with respect to a tangential base point.
	\item Taking the $p$-adic étale realisation yields the $\bQ_p$-pro-unipotent étale fundamental group:
	\[ \real_{\et,p}(\pi_1^{\mot}(X,0)) = \pi_1^{\et,\bQ_p}(X_{\overline{K}},0). \]
	The latter is defined as the Tannaka group of the category of unipotent $\bQ_p$-local systems on $X_{\overline{K}}$. It carries a natural action by $G_{K}$ induced by the action on $X_{\overline{K}}$.
	\item Taking the de Rham realisation yields the unipotent de Rham fundamental group:
	\[ \real_{\dR}(\pi_1^{\mot}(X,0)) = \pi_1^{\dR}(X,0). \]
	The latter is defined as the Tannaka group of the category of unipotent vector bundles with connection on~$X$.
	\item Taking the filtered $\varphi$ realisation for $v \not\in S$ yields the crystalline fundamental group
	\[ \real_{F\varphi,v}(\pi_1^{\mot}(X,0)) = \pi_1^{\cris}(X_{\bF_v},0). \]
	The latter is defined as the Tannaka group of the category of unipotent isocrystals on $X_{\bF_v}$, where $\bF_v$ denotes the residue field of~$v$. The Hodge filtration comes from the crystalline/de Rham comparison isomorphism \cite{chiarelotto-le-stum} 
	\[ \pi_1^{\cris}(X_{\bF_v},0) \otimes_{K_{v,0}} K_v \cong \pi_1^{\dR}(X,0) \otimes_K K_v. \]
\end{enumerate}

It is possible to use other base points in the definition of the motivic fundamental group. More generally, one can define a motivic \emph{path space} from $x$ to $y$,
\[ \pi_1^{\mot}(X; x,y), \]
whenever $x$ and $y$ are $S$-integral base points of~$X$.

\begin{defn}
	\label{def base point}
	An \emph{$S$-integral base point} of $X = \bP^1_{\cO_{K,S}} \smallsetminus \{0,1,\infty\}$ is either an $S$-integral point $x \in X(\cO_{K,S})$ or a nowhere vanishing section of the tangent bundle over $\cO_{K,S}$ at one of the three cusps.
\end{defn}

For $S$-integral base points $x$ and $y$, the motivic path space $\pi_1^{\mot}(X; x,y)$ is a pro-affine scheme in $\MT(\cO_{K,S}, \bQ)$. These path spaces form the \emph{motivic fundamental groupoid} of~$X$. The groupoid structure consists of:
\begin{enumerate}
	\item constant paths $\ast = \Spec(\bQ) \to \pi_1^{\mot}(X, x)$;
	\item path composition $\pi_1^{\mot}(X;y,z) \times \pi_1^{\mot}(X;x,y) \to \pi_1^{\mot}(X;x,z)$;
	\item inversion $\pi_1^{\mot}(X;x,y) \to \pi_1^{\mot}(X;y,x)$.
\end{enumerate}
All these maps are motivic, i.e., morphisms of pro-affine schemes in $\MT(\cO_{K,S}, \bQ)$. They satisfy the expected properties: associativity of path composition, cancellation of inverses, etc.

\subsection{Definition of the motivic Selmer scheme}
\label{sec:motivic Selmer scheme}

Let 
\[ \pi_1^{\mot}(X,0) \twoheadrightarrow \Pi \]
be a quotient of the motivic fundamental group in $\MT(\cO_{K,S}, \bQ)$. As discussed in §\ref{sec:alg_geom_in_tannakian_cat}, we have a notion of $\Pi$-torsors, intrinsic to the category of mixed Tate motives. Namely, a \emph{$\Pi$-torsor} in $\MT(\cO_{K,S}, \bQ)$ over a $\bQ$-algebra~$R$ is an affine $\bQ$-scheme~$P$ in $\MT(\cO_{K,S},\bQ)$, faithfully flat over~$R$, with a right $\Pi$-action $P \times \Pi \to P$ such that the map $P \times \Pi \to P \times_R P$, $(p,\gamma) \mapsto (p, p\gamma)$ is an isomorphism. Here, $R$ is viewed as a ring in $\MT(\cO_{K,S}, \bQ)$ by taking sums of the unit object. The motivic Selmer scheme is now defined as the moduli space of $\Pi$-torsors in $\MT(\cO_{K,S}, \bQ)$. More precisely:

\begin{defn}
	\label{def:motivic_selmer_scheme}
	The \emph{motivic Selmer scheme} of the thrice-punctured line~$X$ for the fundamental group quotient~$\Pi$ is the affine $\bQ$-scheme representing the functor on $\bQ$-algebras
	\begin{equation}
		\label{eq:selmer-functor}
		\Sel_{S,\Pi}^{\mot}(X) \colon R \mapsto \{\text{$\Pi$-torsors over~$R$ in $\MT(\cO_{K,S},\bQ)$\}}/\text{iso}.
	\end{equation}
\end{defn}

The representability of this functor by an affine $\bQ$-scheme is proved in \cite{brown:integral_points}. In Section~\ref{sec:representability} below, we will give a different proof showing the more precise statement that the motivic Selmer functor is representable by a (possibly infinite-dimensional) \emph{affine space} over~$\bQ$, and we construct a set of coordinates on this space which can be used in explicit Chabauty--Kim computations.

\subsection{Independence of the base point}
\label{sec:base-point}

In the definition of the Selmer scheme above we have used the particular (tangential) base point $0 = \vec{1}_0$ for the motivic fundamental group $\pi_1^{\mot}(X, 0)$. In this subsection we show that the Selmer scheme is independent of the choice of base point up to canonical isomorphism.
To make this precise, consider a general setting where $X_K/K$ is a smooth curve over a number field $K$, with smooth $S$-integral model $X/\cO_{K,S}$ for a finite set of primes $S$ of $K$, and $\pi_1(X;-,-)$ denotes any theory of fundamental groupoids in a Tannakian category $\cT = (\cT, \otimes, \one)$ over a field $F$. By this we mean the following data:
\begin{enumerate}
 	\item for any two $S$-integral base points $b, b'$ of $X$ (possibly tangential), a pro-affine $\cT$-scheme $\pi_1(X;b,b')$ (\emph{path space});
 	\item laws of composition $\pi_1(X;b',b'') \times \pi_1(X;b,b') \to \pi_1(X;b,b'')$, $(\gamma',\gamma) \mapsto \gamma'\gamma$;
 	\item constant paths $1_b\colon \Spec(\one) \to \pi_1(X,b) \coloneqq \pi_1(X;b,b)$;
 	\item path inversion $\inv \colon \pi_1(X;b,b') \to \pi_1(X;b',b)$.
\end{enumerate}
These data are subject to natural associativity, identity and inversion conditions. We moreover require the groupoid to be \emph{transitive}, i.e., the path objects $\pi_1(X;b,b')$ should be faithfully flat over the point $\Spec(\one)$. Then $\pi_1(X;b,b')$ is a $\pi_1(X,b')$-$\pi_1(X,b)$-bitorsor in~$\cT$.

Examples of Tannakian fundamental groupoids include the following:
\begin{enumerate}
	\item the Betti fundamental groupoid $\pi_1^{\Betti}(X_{\sigma}(\bC); -,-)$ in $\cT = \bQ\mhyphen\Vect_{\rf}$ associated to a complex embedding $\sigma\colon K \hookrightarrow \bC$;
	\item the $\bQ_p$-pro-unipotent étale fundamental groupoid $\pi_1^{\et,\bQ_p}(X_{\overline{K}}; -,-)$ in $\cT = \Rep_{\bQ_p}(G_K)$;
	\item the de Rham fundamental groupoid $\pi_1^{\dR}(X;-,-)$ in $\cT = K\mhyphen\Vect_{\rf}$;
	\item the crystalline fundamental groupoid $\pi_1^{\cris}(X_{\bF_v}; -,-)$ in $\cT = \MF_{K_v}^{\varphi,\adm}$ for a prime $v \not\in S$ of~$K$.
\end{enumerate}
In the case of $X = \bP^1 \smallsetminus \{0,1,\infty\}$ over $\cO_{K,S}$ we also have:
\begin{enumerate}[resume]
	\item the motivic fundamental groupoid $\pi_1^{\mot}(X;-,-)$ in $\cT = \MT(\cO_{K,S}, \bQ)$.
\end{enumerate}

Suppose $b$ and $c$ are two $S$-integral base points of $X/\cO_{K,S}$. Any $\pi_1(X,b)$-torsor $P$ in $\cT$ can be twisted by the path bitorsor $\pi_1(X;c,b)$ to yield a $\pi_1(X,c)$-torsor $P \overset{\pi_1(X,b)}{\times} \pi_1(X;c,b)$, defined as the coequaliser of the maps
\begin{align*}
	P \times \pi_1(X,b) \times \pi_1(X;c,b) &\rightrightarrows P \times \pi_1(X;c,b),\\
	(p,\delta,\gamma) &\mapsto (p\delta, \gamma) \quad \text{resp.} \quad (p,\delta \gamma).
\end{align*}
 This construction makes sense for $\pi_1(X,b)$-torsors over $\bQ$-algebras in the Tannakian category~$\cT$ and thus defines a canonical isomorphism of moduli spaces of torsors
 \begin{equation}
	 	\label{eq: change of base point iso}
	 	\Tors_{\cT}(\pi_1(X,b)) \cong \Tors_{\cT}(\pi_1(X,c)),
	 \end{equation}
 expressing the base point independence of the Selmer scheme for the full fundamental group.
 
 Now suppose $\pi_1(X,b) \twoheadrightarrow \Pi_b$ is a quotient in $\cT$ with kernel $N_b$. We construct a corresponding quotient of $\pi_1(X,c)$ by conjugation with the path space $\pi_1(X;b,c)$, in a sense to be made precise. Given any ring $A$ in $\cT$ and an $A$-valued path $\gamma \in \pi_1(X;b,c)(A)$, conjugation by $\gamma$ defines an isomorphism of affine $\cT$-group schemes over $A$
 \[ \gamma(-)\gamma^{-1}\colon \pi_1(X,b) \otimes A \cong \pi_1(X,c) \otimes A. \]
 The image of $N_b \otimes A$ under this isomorphism is a normal subgroup $\gamma (N_b \otimes A) \gamma^{-1}$ of $\pi_1(X,c) \otimes A$.
 
 \begin{lemma}
	 	\label{lem: conjugate subgroup}
	 	There exists a unique normal subgroup $N_c \subseteq \pi_1(X,c)$ such that for every $\cT$-ring $A$ and any $A$-valued path $\gamma \in \pi_1(X;b,c)(A)$, we have
	 	\[ \gamma (N_b \otimes A)\gamma^{-1} = N_c \otimes A. \]
	 \end{lemma}
 
 \begin{proof}
	 	If $\gamma$ and $\gamma'$ are two $A$-valued paths, then $\gamma(N_b \otimes A) \gamma^{-1}$ and $\gamma' (N_b \otimes A) \gamma'^{-1}$ are conjugate via the loop $\gamma'\gamma^{-1} \in \pi_1(X,c)(A)$, hence by normality they coincide. Thus, $\gamma (N_b \otimes A) \gamma^{-1} \subseteq \pi_1(X,c) \otimes A$ does not depend on the choice of $\gamma$.
	 	
	 	Assume that $A$ is faithfully flat over $\one$. Then applying the same argument to the $A \otimes A$-valued paths $\pr_1^* \gamma$ and $\pr_2^* \gamma$ (where $\pr_i\colon \Spec(A) \times \Spec(A) \to \Spec(A)$ are the two projections) and using faithfully flat descent shows that $\gamma (N_b \otimes A) \gamma^{-1} \subseteq \pi_1(X,c) \otimes A$ arises by base change from a normal subgroup of $\pi_1(X,c)$.
	 	
	 	Choose $A_0 = \cO(\pi_1(X;b,c))$ and let $\gamma_0 \in \pi_1(X;b,c)(A_0)$ be the tautological path. We assumed the fundamental groupoid in $\cT$ to be transitive, thus $\cO(\pi_1(X;b,c))$ is faithfully flat over $\one$. So there exists a unique normal subgroup $N_c \subseteq \pi_1(X,c)$ such that $\gamma_0(N_b \otimes A_0)\gamma_0^{-1} = N_c \otimes A_0$.
	 	
	 	If $A$ is any other $\cT$-ring and $\gamma$ is an $A$-valued path, then, since $A \to A \otimes A_0$ is faithfully flat, the claimed equality $\gamma(N_b \otimes A)\gamma^{-1} = N_c \otimes A$ can be checked after base change to $A \otimes A_0$, where it follows from the fact that conjugating by the $A \otimes A_0$-valued paths $\gamma \otimes A_0$ or $A \otimes \gamma_0$ results in the same subgroup, i.e., in $N_c \otimes A \otimes A_0$.
	 \end{proof}
 
 \begin{defn}
	 	\label{def: conjugate subgroup in groupoid}
	 	Let $N_b \subseteq \pi_1(X,b)$ be a normal subgroup in $\cT$. The unique corresponding normal subgroup $N_c \subseteq \pi_1(X,c)$ from \Cref{lem: conjugate subgroup} is called the \emph{conjugate subgroup}. For a quotient $\pi_1(X,b) \twoheadrightarrow \Pi_b$ in~$\cT$ with kernel~$N_b$, the \emph{corresponding quotient} $\pi_1(X,c) \twoheadrightarrow \Pi_c$ is the one whose kernel~$N_c$ is conjugate to~$N_b$.
	 \end{defn}
 
 \begin{prop}
	 	\label{thm: torsors change of base point}
	 	Let $\pi_1(X,b) \twoheadrightarrow \Pi_b$ and $\pi_1(X,c) \twoheadrightarrow \Pi_c$ be corresponding quotients of the fundamental groups in $\cT$ with respective kernels $N_b$ and $N_c$. Then there is a natural isomorphism $\Tors(\Pi_b) \cong \Tors(\Pi_c)$ making the following square commute:
	 	\[
	 	\begin{tikzcd}
		 		\Tors_{\cT}(\pi_1(X,b)) \rar["\eqref{eq: change of base point iso}"] \dar & \Tors_{\cT}(\pi_1(X,c)) \dar \\
		 		\Tors_{\cT}(\Pi_b) \rar[dashed] & \Tors_{\cT}(\Pi_{c}).
		 	\end{tikzcd}
	 	\]
	 \end{prop}
 
 \begin{proof}
 	We claim that the quotient of the path space $\pi_1(X;c,b)$ by $N_c$ on the right is canonically isomorphic to the quotient by $N_b$ on the left:
 	\begin{equation}
	 		\label{eq: path space quotient}
	 		\pi_1(X;c,b)/N_c \cong N_b\backslash\pi_1(X;c,b).
	 	\end{equation}
 	The quotient $\pi_1(X;c,b)/N_c$ is by definition the coequaliser in the category of affine $\cT$-schemes of the two maps $\pi_1(X;c,b) \times N_c \rightrightarrows \pi_1(X;c,b)$ given by multiplication and first projection, respectively, and similarly for the quotient by~$N_b$ on the left. Consider the following diagram with coequaliser columns:
 	\[
 	\begin{tikzcd}
	 		\pi_1(X;c,b) \times N_c \dar["\pr_1", shift left] \dar["\mathrm{mult}"', shift right] & N_b \times \pi_1(X;c,b) \dar["\pr_2", shift left] \dar["\mathrm{mult}"', shift right] \\
	 		\pi_1(X;c,b) \dar \rar[equals] & \pi_1(X;c,b) \dar \\
	 		\pi_1(X;c,b)/N_c \rar[dashed] & N_b\backslash \pi_1(X;c,b).
	 	\end{tikzcd}
 	\]
 	For any $\cT$-ring $A$ and $A$-valued points $\gamma \in \pi_1(X;c,b)(A)$ and $\delta_c \in N_c(A)$ we have $\gamma \delta_c = \delta_b \gamma$ with $\delta_b = \gamma \delta_c \gamma^{-1}$ an element in the conjugate subgroup $N_b$. Thus, the dashed arrow is well-defined. By symmetry considerations, it is an isomorphism.
 	
 	Denote the quotient~\eqref{eq: path space quotient} of the path space by $\pi_1(X;c,b) \twoheadrightarrow \tensor*[_b]{\Pi}{_{c}}$. Note that $\tensor*[_b]{\Pi}{_{c}}$ carries a left action by $\Pi_b = \pi_1(X,b)/N_b$ and a right action by $\Pi_c = \pi_1(X,c)/N_c$. It is simultaneously the pushforward of $\pi_1(X;c,b)$ as a left torsor along the quotient map $\pi_1(X,b) \twoheadrightarrow \Pi_b$, and as a right torsor along $\pi_1(X,c) \twoheadrightarrow \Pi_c$. In particular, $\tensor*[_b]{\Pi}{_{c}}$ is a $\Pi_b$-$\Pi_c$-bitorsor. Any $\Pi_b$-torsor $P$ can be twisted by $\tensor*[_b]{\Pi}{_{c}}$ to yield a $\Pi_c$-torsor $P \overset{\Pi_b}{\times} \tensor*[_b]{\Pi}{_{c}}$, which defines the claimed natural isomorphism $\Tors(\Pi_b) \cong \Tors(\Pi_c)$.
\end{proof}

\begin{cor}
	\label{selmer-scheme-indep-of-base-point}
	Let $b$ and $c$ be two $S$-integral base points of $\bP^1 \smallsetminus \{0,1,\infty\}$ and let $\pi_1^{\mot}(X,b) \twoheadrightarrow \Pi_b$ and $\pi_1^{\mot}(X,c) \twoheadrightarrow \Pi_c$ be corresponding quotients in $\MT(\cO_{K,S},\bQ)$ with respective kernels~$N_b$ and~$N_c$. Then there is a canonical isomorphism of Selmer schemes
	\[ \Sel_{S,\Pi_b}^{\mot}(X) \cong \Sel_{S,\Pi_c}^{\mot}(X) \]
	given by twisting with the path torsor 
	\[ \tensor*[_b]{\Pi}{_{c}} \coloneqq \pi_1^{\mot}(X;c,b)/N_c = N_b \backslash \pi_1^{\mot}(X;c,b). \]
\end{cor}

\begin{proof}
	This is \Cref{thm: torsors change of base point} applied to the motivic fundamental groupoid of $\bP^1 \smallsetminus \{0,1,\infty\}$ in the Tannakian category $\MT(\cO_{K,S}, \bQ)$.
\end{proof}

\begin{rem}
	Since \Cref{thm: torsors change of base point} is proved for general Tannakian fundamental groupoids, it applies equally to the Bloch--Kato Selmer scheme $\rH^1_{f,S}(G_K, \Pi^{\et})$ of a smooth curve over~$\cO_{K,S}$. The latter can be defined as the moduli space of $\Pi^{\et}$-torsors in the Tannakian category~$\Rep_{\bQ_p}^{f,S}(G_K)$ of $p$-adic representations of~$G_K$ which are unramified at all places $v \not\in S$ not dividing~$p$ and crystalline at all places dividing~$p$. Here, $p$ is a rational prime not divisible by any prime in~$S$, and $\Pi^{\et}$ is a quotient of the $\bQ_p$-pro-unipotent étale fundamental group $\pi_1^{\et,\bQ_p}(X_{\overline{K}}, b)$ in~$\Rep_{\bQ_p}^{f,S}(G_K)$. The base point independence of the Bloch--Kato Selmer scheme was previously shown in~\cite[§2.9]{BDCKW}.
\end{rem}

\subsection{Comparison with the étale Selmer scheme}
\label{sec:comparison}

Let $\pi_1^{\mot}(X,0) \twoheadrightarrow \Pi$ be a quotient of the motivic fundamental group in~$\MT(\cO_{K,S},\bQ)$ and fix a rational prime $p$ not divisible by a prime in~$S$. Then the $p$-adic étale realisation $\Pi^{\et}$ of $\Pi$ is a $G_K$-equivariant quotient of the $\bQ_p$-pro-unipotent étale fundamental group $\pi_1^{\et,\bQ_p}(X_{\overline{K}},0) \twoheadrightarrow \Pi^{\et}$. 

\begin{defn}
	\label{def:etale-selmer-scheme}
	For a $\bQ_p$-algebra~$R$, let $\rH^1_{f,S}(G_K, \Pi^{\et})(R)$ be the set of isomorphism classes of $G_K$-equivariant $\Pi^{\et}$-torsors which are crystalline at all primes dividing~$p$ and unramified at all primes $v\not\in S$ with $v\nmid p$. 
	The \emph{étale Selmer scheme} $\rH^1_{f,S}(G_K,\Pi^{\et})$ is the affine $\bQ_p$-scheme representing the functor $R \mapsto \rH^1_{f,S}(G_K, \Pi^{\et})(R)$.
\end{defn}

The fact that the étale Selmer functor is representable by an affine $\bQ_p$-scheme is shown in \cite{kim:motivic}.
In the case $K = \bQ$ it is proved in \cite[§3]{BKL:chabauty-kim-sc} that the motivic Selmer scheme provides a $\bQ$-structure of the étale Selmer scheme. We briefly sketch how the proof generalises over a general number field. The main technical ingredient is the Ext group comparison theorem \cite[Proposition~A.1]{BKL:chabauty-kim-sc} which is already established in this generality.

If $R$ is a $\bQ_p$-algebra and $P$ is a $\Pi$-torsor over~$R$ in $\MT(\cO_{K,S}, \bQ)$, then the $p$-adic étale realisation $P^{\et}$ is a $G_K$-equivariant $\Pi^{\et}$-torsor over~$R \otimes_{\bQ} \bQ_p$. Base-changing along the multiplication map $R \otimes_{\bQ} \bQ_p \to R$ yields a $\Pi^{\et}$-torsor over~$R$. This construction defines a morphism of $\bQ_p$-schemes from the motivic to the étale Selmer scheme:
\begin{equation}
	\label{eq:motivic-to-etale}
	\Sel_{S,\Pi}^{\mot}(X)_{\bQ_p} \to \rH^1_{f,S}(G_K, \Pi^{\et}).
\end{equation}

\begin{prop}
	\label{motivic-etale-comparison}
	The map \eqref{eq:motivic-to-etale} is an isomorphism.
\end{prop}

\begin{proof}
	
	Let $\Rep_{\bQ_p}^{\MT,S}(G_K)$ denote the category of $\bQ_p$-linear mixed Tate representations of $G_K$ which are unramified outside $S \cup \{\fp \mid p\}$ and crystalline at all primes dividing~$p$. This is a mixed Tate category over $\bQ_p$ and $p$-adic étale realisation defines a Tate functor 
	\[ \rho_{\et}\colon \MT(\cO_{K,S},\bQ) \to \Rep_{\bQ_p}^{\MT,S}(G_K), \]
	thus inducing a morphism of Tannakian fundamental groups $\rho_{\et}^*\colon G_{K,S}^{\MTR} \to G_{S,\bQ_p}^{\MT} \coloneqq G_S^{\MT} \otimes_{\bQ} \bQ_p$ \cite[§3.2]{BKL:chabauty-kim-sc}. By \cite[Proposition~A.1]{BKL:chabauty-kim-sc}, the map
	\[ \Ext^1_{\MT(\cO_{K,S},\bQ)}(\bQ(0), \bQ(n)) \otimes_{\bQ} \bQ_p \to \Ext^1_{\Rep_{\bQ_p}^{\MT,S}(G_K)}(\bQ_p(0),\bQ_p(n)) \]
	induced by $\rho_{\et}$ is an isomorphism for all $n > 0$, which implies by \cite[Lemma~3.6]{BKL:chabauty-kim-sc} that the map $\rho_{\et}^*\colon G_{K,S}^{\MTR} \to G_{S,\bQ_p}^{\MT}$ is an isomorphism. Then we get the desired isomorphism~\eqref{eq:motivic-to-etale} via
	\[ \Sel_{S,\Pi}^{\mot}(X)_{\bQ_p} = \rH^1(G_{S,\bQ_p}^{\MT}, \Pi^{\omega}_{\bQ_p}) \xrightarrow{\sim} \rH^1(G_{K,S}^{\MTR}, \Pi^{\omega}_{\bQ_p}) \cong \rH^1_{f,S}(G_K, \Pi^{\et}), \]
	where the last isomorphism comes from the fact that a $G_K$-equivariant $\Pi^{\et}$-torsor is automatically mixed Tate, see \cite[§3.3]{BKL:chabauty-kim-sc}.
\end{proof}

\section{The motivic Selmer scheme via algebraic group cohomology}
\label{sec:motivic-selmer-scheme-via-algebraic-cocycles}

The definition of the motivic Selmer scheme as a moduli space of torsors is elegant but not very explicit. As a first step towards making this space more amenable to computations, we give a description in terms of non-abelian algebraic group cohomology of the mixed Tate fundamental group.

\subsection{The mixed Tate Galois group of an integer scheme}
\label{sec:MT Galois group}

The category $\MT(\cO_{K,S},\bQ)$ has a canonical $\bQ$-linear fibre functor $\omega$ given by
\begin{equation}
	\label{eq: canonical fibre functor}
	\omega(M) = \bigoplus_{i \in  \bZ} \omega_i(M)
\end{equation}
where
\[ \omega_i(M) \coloneqq \Hom_{\MT(\cO_{K,S},\bQ)} (\bQ(-i), \gr_{2i}^W(M)). \]
We call the grading of $\omega(M)$ provided by \eqref{eq: canonical fibre functor} the grading by \emph{half-weight}, i.e., $\omega_i(M)$ sits in half-weight~$i$. For an object $M$ of $\MT(\cO_{K,S}, \bQ)$, we denote its canonical realisation $\omega(M)$ also by~$M^\omega$. If $K = \bQ$, then $\omega$ coincides with the de Rham realisation functor. In general, the de Rham realisation functor is obtained from~$\omega$ by base change from $\bQ$ to $K$ \cite[Prop.~2.10]{deligne-goncharov}. In other words, $\omega$ provides a $\bQ$-structure for $\real_{\dR}$.

\begin{defn}
	\label{def: mixed Tate Galois group}
	The \emph{mixed Tate Galois group} $G_S^{\MT}$ of $\cO_{K,S}$ is the Tannakian fundamental group of $\MT(\cO_{K,S}, \bQ)$ with respect to the canonical fibre functor:
	\[ G_S^{\MT} \coloneqq \underline{\Aut}^{\otimes}(\omega). \]
\end{defn}

The mixed Tate Galois group $G_S^{\MT}$ is thus a pro-algebraic group over~$\bQ$. More precisely, it is a semidirect product of $\bG_m$ by a free pro-unipotent group. Namely, there is a surjective homomorphism $G_S^{\MT} \twoheadrightarrow \bG_m$ given by the action of $G_S^{\MT}$ on $\bQ(-1)$.\footnote{This is the opposite of the convention in \cite{deligne-goncharov}, where the homomorphism is given by the action on $\bQ(1)$.} The kernel, a pro-unipotent group denoted by $U_S^{\MT}$, is the subgroup of $G_S^{\MT}$ acting trivially on weight-graded pieces. The group $U_S^{\MT}$ is called the \emph{unipotent mixed Tate Galois group}. We have a short exact sequence:
\[ 1 \lto U_S^{\MT} \lto G_S^{\MT} \lto \bG_m \lto 1. \]
The sequence admits a canonical splitting $\bG_m \to G_S^{\MT}$ by associating to $\lambda \in \bG_m$ the automorphism of the fibre functor $\omega(M) = \bigoplus_{i \in  \bZ} \omega_i(M)$ that acts on $\omega_i(M)$ as multiplication by $\lambda^i$ for all $i \in \bZ$. In this way, $G_S^{\MT}$ becomes a semidirect product:
\[ G_S^{\MT} = U_S^{\MT} \rtimes \bG_m. \]

It follows from the vanishing of the second Ext groups $\Ext^2_{\MT(\cO_{K,S}, \bQ)}(\bQ(0), \bQ(n))$ for $n \in \bZ$ \cite[Prop.~2.3]{deligne-goncharov} that $U_S^{\MT}$ is a \emph{free} pro-unipotent group. Equivalently, the Lie algebra $\Lie(U_S^{\MT})$ is free pro-nilpotent. We discuss how to obtain a set of free generators in §\ref{sec:generators of unipotent MT Galois group} below.

\subsection{Tannakian torsors via algebraic group cohomology}
\label{sec:torsors_in_tannakian_categories}

Consider a general Tannakian category $\cT = (\cT, \otimes, \one)$ over a field~$F$. Let $U$ be a pro-unipotent group in~$\cT$. Our aim is to explain how $U$-torsors in~$\cT$ are parametrised by algebraic group cohomology. 
Assume that $\cT$ is neutral with fibre functor $\omega\colon \cT \to F\mhyphen\Vect_{\rf}$. Denote by $G_{\cT} = \underline{\Aut}^{\otimes}(\omega)$ the Tannakian fundamental group of $\cT$ with respect to $\omega$. It acts on the coordinate ring $\cO(\omega(U))$ and hence on the group $\omega(U)$ itself. 

\begin{defn}
	\label{def: cocycle functor}
	Let $R$ be an $F$-algebra. An algebraic map $c\colon G_{\cT,R} \to \omega(U)_R$ is a \emph{cocycle} if it satisfies
	\[ c(g_1 g_2) = c(g_1) \cdot g_1(c(g_2)) \]
	for all $g_1, g_2 \in G_{\cT}$ valued in an $R$-algebra. The set of all cocycles over $R$ is denoted by $\rZ^1(G_{\cT,R}, \omega(U)_R)$. 
	The \emph{functor of cocycles} is defined on $F$-algebras by
	\[ \underline{\rZ}^1(G_{\cT}, \omega(U)): R \mapsto \rZ^1(G_{\cT,R}, \omega(U)_R). \]
\end{defn}

For any $F$-algebra $R$, the group $\omega(U)(R)$ acts from the right on cocycles via
\[ (c.u)(g) = u^{-1} \cdot c(g) \cdot g(u) \]
for $c \in \rZ^1(G_{\cT,R}, \omega(U)_R)$, $u \in \omega(U)(R)$ and $g \in G_{\cT,R}$. Denote by
\[ \rH^1(G_{\cT,R}, \omega(U)_R) \coloneqq  \rZ^1(G_{\cT,R}, \omega(U)_R)/\omega(U)(R) \]
the pointed set of cohomology classes of cocycles and by 
\[ \underline{\rH}^1(G_{\cT}, \omega(U))\colon R \mapsto \rH^1(G_{\cT,R}, \omega(U)_R) \]
the cohomology functor on $F$-algebras.

\begin{prop}
	\label{thm: torsors in tannakian categories via cocycles}
	There is a natural isomorphism of functors on $F$-algebras
	\[ \bigl(\text{isomorphism classes of $U$-torsors in $\cT$}\bigr) \cong \underline{\rH}^1(G_{\cT}, \omega(U)). \]
\end{prop}

\begin{proof}
	Let $R$ be an $F$-algebra and let $P$ be a $U$-torsor over $R$ in $\cT$. One can check that the realisation $\omega(P)$ is an $\omega(U)$-torsor over $R$ in the usual sense. Since~$U$ is pro-unipotent, the torsor $\omega(P)$ is trivial. Choose a point $p_0 \in \omega(P)(R)$. This determines a cocycle $c \colon G_{\cT,R} \to \omega(U)_R$ via 
	\[ g(p_0) = p_0 \cdot c(g) \]
	for all $g \in G_{\cT}$ valued in an $R$-algebra. A different choice of $p_0$ is of the form $p_0 u$ for a unique $u \in \omega(U)(R)$. Since the action map $\omega(P) \times \omega(U) \to \omega(P)$ is induced by a map in $\cT$, it is $G_{\cT}$-equivariant. Hence we have
	\[ g(p_0 u) = g(p_0) g(u) = p_0 c(g) g(u) = (p_0 u) \cdot u^{-1}c(g) g(u), \]
	for all $g \in G_{\cT,R}$, so the cocycles associated to $p_0$ and $p_0 u$ are cohomologous and the class of the cocycle is well-defined. A similar calculation shows that isomorphic torsors give rise to the same cocycle class.
	
	Conversely, let $c$ be any cocycle over $R$. Let $\tilde{P}_c$ be the trivial $\omega(U)$-torsor $\omega(U)_R$ over $R$ whose left $G_{\cT}$-action $t_c\colon G_{\cT} \to \Aut(\tilde{P}_c)$ is obtained from the usual one by twisting with the cocycle $c$:
	\[ t_c(g)(p) \coloneqq c(g) g(p). \]
	This gives $\tilde{P}_c$ the structure of a $\cT$-scheme, i.e., there is an affine $F$-scheme $P_c$ in~$\cT$ such that $\omega(P_c) = \tilde{P}_c$. Moreover, the structural morphism $\tilde{P}_c \to \Spec(R)$ and the action map $\tilde{P}_c \times \omega(U) \to \tilde{P}_c$ are $G_{\cT}$-equivariant, hence they are morphisms of $\cT$-schemes. The map of $\cT$-schemes $P_c \to \Spec(R)$ is faithfully flat and the map $P_c \times U \to P_c \times_{R_\cT} P_c$, $(p,u) \mapsto (p,pu)$ is an isomorphism. This can be checked after applying the fibre functor by exactness and faithfulness of $\omega$. Hence, $P_c$ is a $U$-torsor over $R$ in $\cT$. 
	
	The constructions are inverse to each other. Indeed, denote by $c_{p_0}$ the cocycle defined by $p_0 \in \omega(P)(R)$ for a $U$-torsor $P$ over $R$. When $P = P_c$ is the torsor associated to a cocycle~$c$, we can choose $p_0 = 1 \in \omega(P_c)(R) = U(R)$ and have
	\[ g(p_0) = t_c(g)(1) = c(g) g(1) = c(g) = 1 \cdot c(g) = p_0 \cdot c(g), \]
	so that $c_{p_0} = c$. Conversely, for any $P$ and $p_0 \in \omega(P)(R)$, the map $\omega(U)_R \to \omega(P)$ given by $u \mapsto p_0 u$ is $G_{\cT}$-equivariant for the $c_{p_0}$-twisted action on $\omega(U)$ and therefore comes from a map $P_{c_{p_0}} \to P$ in $\cT$. This is a map of right $U$-torsors over $R$ and hence automatically an isomorphism: $P_{c_{p_0}} \cong P$.
\end{proof}

\begin{rem}
	For a pro-algebraic group $G$ over a field~$F$ acting on a pro-unipotent group~$U$ one can define the notion of \emph{$G$-equivariant $U$-torsor}. Such an object is given by a right $U$-torsor $P$ with a $G$-action $G \times P \to P$ that is compatible with the $G$-action on~$U$: $g.(pu) = (g.p)(g.u)$. One has an equivalence of groupoids between $G$-equivariant $U$-torsors on the one hand, and $U$-torsors in the Tannakian category $\Rep(G)$ on the other hand. By \Cref{thm: torsors in tannakian categories via cocycles} both are parametrised by the algebraic group cohomology set $\rH^1(G, U)$. This is analogous to the torsor interpretation of first non-abelian group cohomology of a profinite group, cf.\ \cite[Prop.~(1.2.3)]{nsw}.
\end{rem}

\subsection{The Selmer scheme via algebraic group cohomology}
\label{sec:Sel-via-group-cohomology}

Recall that in \Cref{def:motivic_selmer_scheme} the motivic Selmer scheme $\Sel_{S,\Pi}^{\mot}(X)$ for a quotient of the motivic fundamental group $\pi_1^{\mot}(X,0) \twoheadrightarrow \Pi$ was defined as the moduli space of $\Pi$-torsors in the Tannakian category~$\MT(\cO_{K,S}, \bQ)$. 
\Cref{thm: torsors in tannakian categories via cocycles} says that those torsors are parametrised by non-abelian algebraic group cohomology. This gives us a cohomological description of the motivic Selmer scheme:

\begin{thm}
	\label{thm:Selmer scheme via H1}
	There is a canonical isomorphism of functors on $\bQ$-algebras
	\[ \Sel_{S,\Pi}^{\mot}(X) \cong \underline{\rH}^1(G_S^{\MT}, \Pi^{\omega}). \]
\end{thm}

\begin{proof}
	This is \Cref{thm: torsors in tannakian categories via cocycles} applied to the Tannakian category $\MT(\cO_{K,S}, \bQ)$ and pro-unipotent group $\Pi$.
\end{proof}

\section{The motivic Selmer scheme via group cocycles}
\label{sec: motivic Selmer scheme via cocycles}

As a next step towards a more explicit description of the motivic Selmer scheme, we show that every cohomology class in $\rH^1(G_S^{\MT}, \Pi^\omega)$ is represented by a unique cocycle which is trivial on the subgroup $\bG_m \subseteq G_S^{\MT}$. The restriction of such a cocycle to the unipotent part $U_S^{\MT}$ is then $\bG_m$-equivariant. We will use this to identify the motivic Selmer scheme with the space of $\bG_m$-equivariant cocycles $U_S^{\MT} \to \Pi^\omega$, see \Cref{thm: equivariant cocycles description of Selmer scheme}.

\subsection{The $\bG_m$-invariant point of a \texorpdfstring{$\Pi$}{Π}-torsor}

Fix a quotient $\pi_1^{\mot}(X,0) \twoheadrightarrow \Pi$ of the motivic fundamental group in $\MT(\cO_{K,S}, \bQ)$. Its canonical realisation $\Pi^{\omega}$ is then a pro-unipotent group over~$\bQ$ with an action by the mixed Tate motivic Galois group $G_S^{\MT}$. In particular, the subgroup $\bG_m$ acts on $\Pi^{\omega}$. Let $(\Pi^{\omega})^{\bG_m} \subseteq \Pi^{\omega}$ be the subfunctor of $\bG_m$-invariants, defined on $\bQ$-algebras~$R$.

\begin{lemma}
	\label{thm: fundamental group has trivial invariants}
	The group $\Pi^{\omega}$ has trivial $\bG_m$-invariants:
	\[ (\Pi^{\omega})^{\bG_m} = \{1\}. \]
\end{lemma}

\begin{proof}
	We give an argument based on the weights of the abelian subquotients of~$\Pi$. Denote by $\Pi_n$ the $n$-step nilpotent quotient of $\Pi$ and let $V_n$ be the kernel of the projection $\Pi_n \twoheadrightarrow \Pi_{n-1}$, so that we have short exact sequences
	\begin{equation}
		\label{eq: iterated extension of vector groups}
		1 \to V_n \to \Pi_n \to \Pi_{n-1} \to 1
	\end{equation}
	for $n \geq 1$. 
	The groups $V_n$ are abelian pro-unipotent groups, i.e. pro-vector groups, in $\MT(\cO_{K,S}, \bQ)$. They can be viewed simply as pro-objects of $\MT(\cO_{K,S}, \bQ)$. For $n = 1$, the group $V_1$ is the abelianisation $V_1 = \Pi_1 = \Pi^{\ab}$. It is a quotient of the full abelianised motivic fundamental group $\pi_1^{\mot}(X,0)^{\ab} = \rH_1^{\mot}(X)$, which is isomorphic as a mixed Tate motive to $\bQ(1) \times \bQ(1)$. One way of seeing this is to consider the two maps $z \mapsto z$ and $z \mapsto 1-z$ from $X = \bP^1 \smallsetminus \{0,1,\infty\}$ to $\bG_m$ and check on Betti realisations that the induced map of motives 
	\[ \rH_1^{\mot}(X) \to \rH_1^{\mot}(\bG_m) \times \rH_1^{\mot}(\bG_m) = \bQ(1) \times \bQ(1)\]
	is an isomorphism. As a quotient of $\bQ(1) \times \bQ(1)$, which is pure of weight~$-2$, the group $V_1$ is also pure of weight~$-2$. In particular, $(V_1^{\omega})^{\bG_m} = 0$. For arbitrary $n \geq 1$, the nested commutator map defines a surjection of mixed Tate motives
	\[ V_1^{\otimes n} \twoheadrightarrow V_n, \quad x_1 \otimes \ldots \otimes x_n \mapsto [x_1, [x_2,[\ldots,x_n]]]. \]
	This implies that $V_n$ is pure of weight~$-2n$ and in particular satisfies $(V_n^{\omega})^{\bG_m} = 0$ as well. Now an inductive argument using the short exact sequences~\eqref{eq: iterated extension of vector groups} yields that $\Pi_n^{\omega} = \{1\}$ for all $n \geq 1$, and hence that $(\Pi^{\omega})^{\bG_m} = \varprojlim\, (\Pi_n^{\omega})^{\bG_m} = \{1\}$.
\end{proof}

The goal of this subsection is to show that not only $\Pi$ itself but in fact any $\Pi$-torsor contains a unique $\bG_m$-invariant point:

\begin{thm}
	\label{thm: existence of invariant point}
	Let $P$ be a $\Pi$-torsor in $\MT(\cO_{K,S},\bQ)$ over a $\bQ$-algebra $R$. Then the canonical realisation $P^{\omega}$ contains a unique $\bG_m$-invariant $R$-valued point.
\end{thm}

For the proof of \Cref{thm: existence of invariant point} we need a lemma in non-abelian group cohomology. To state the lemma, recall that a pro-unipotent group is finitely generated if its group of rational points admits a Zariski-dense finitely generated subgroup. This implies that the descending central series quotients are finite-dimensional. Recall moreover that a linear algebraic group over a field is called \emph{linearly reductive} if all its linear representations are semisimple. In characteristic zero, reductive groups are linearly reductive.

\begin{lemma}
	\label{thm: cohomology of reductive group on unipotent group}
	Let $G$ be a linearly reductive group over a field $F$ acting on a finitely generated pro-unipotent group $U$. Then
	\[ \underline{\rH}^1(G, U) = 1. \]
\end{lemma}

\begin{proof}
	Let $R$ be an $F$-algebra. We are claiming that $\rH^1(G_R, U_R) = 1$.
	
	\emph{Step 1: Reduction to the finite-dimensional case.}
	Let $U = U^{(1)} \supseteq U^{(2)} \supseteq \ldots$ be the lower central series of $U$ with corresponding quotients $U_n \coloneqq U/U^{(n+1)}$. The $U_n$, being characteristic quotients of $U$, inherit the $G$-action. They are finite-dimensional since $U$ is finitely generated. Assuming the finite-dimensional case, we have $\rH^1(G_R, (U_n)_R) = 1$ for all $n$. Let $c \colon G_R \to U_R$ be a cocycle. Its composition with the quotient map $p_n \colon U_R \to (U_n)_R$ is cohomologous to the trivial cocycle, so there exists $u_n \in U_n(R)$ such that $p_n(c(g)) = u_n^{-1} g(u_n)$ for all $g \in G_R$. Let $T_n$ be the set of all such $u_n$. The $T_n$ form an $\bN$-indexed inverse system of nonempty sets. We claim that the transition maps $T_n \to T_{n-1}$ are surjective. Indeed, given $u_{n-1} \in T_{n-1}$, let $\tilde u_n \in U_n(R)$ be any lift and form the cocycle 
	\[ \tilde c_n\colon G_R \to (U_n)_R, \quad g \mapsto \tilde u_n \, p_n(c(g)) \, g(\tilde u_n^{-1}). \]
	Since $\tilde u_n \mapsto u_{n-1}$, the cocycle takes in fact values $(V_n)_R$, where $V_n \coloneqq U^{(n)}/U^{(n+1)}$ is the kernel of $U_n \twoheadrightarrow U_{n-1}$. Again using the finite-dimensional case we have $\rH^1(G_R, (V_n)_R) = 1$, so that $\tilde c_n(g) = v_n^{-1} g(v_n)$ for some $v_n \in V_n(R)$. Then we can choose $u_n \coloneqq v_n \tilde u_n$ which is contained in $T_n$ and still lifts $u_n$. So the transition maps $T_n \twoheadrightarrow T_{n-1}$ are surjective, hence their limit is nonempty. Let $u = \lim_n u_n \in U(R)$ be an element of the limit, then we have $c(g) = u^{-1} g(u)$ for all $g \in G_R$ and hence $c$ is cohomologous to the trivial cocycle.
	
	\emph{Step 2: Reduction to the abelian case.}
	By Step~1, we may assume that $U$ is finite-dimensional. This implies that the lower central series is finite, i.e.\ $U = U_N$ for $N \gg 0$. Assuming the abelian case, we show by induction on $n$ that $\rH^1(G_R, (U_n)_R) = 1$. The case $n = 0$ is trivial since $U_0 = 1$. For $n \geq 1$, the short exact sequence
	\[ 1 \lto V_n \lto U_n \lto U_{n-1} \lto 1, \]
	leads to an exact sequence of pointed sets
	\[ \rH^1(G_R, (V_n)_R) \lto \rH^1(G_R, (U_n)_R) \lto \rH^1(G_R, (U_{n-1})_R). \]
	The outer terms are trivial since we are assuming the abelian case and the induction hypothesis, so the middle term is trivial as well.
	
	\emph{Step 3: Reduction to the case $R = F$.}
	By Step~2, we may assume that $V \coloneqq U$ is abelian, i.e., a vector group. The $G$-action on $V$ is equivalent to a representation of $G$ on the finite-dimensional vector space $V(F)$. The set $\rH^1(G_R, V_R)$ is a Hochschild cohomology group, as defined in \cite[Appendix B.2]{conrad:reductive_group_schemes}.
	In particular, there is a cochain complex $\rC^{\bullet}(G_R, V_R)$ of $R$-modules such that 
	\[ \rH^n(G_R, V_R) = \rH^n(\rC^{\bullet}(G_R, V_R)). \]
	The $R$-module $\rC^n(G_R, V_R)$ is given by the algebraic maps $G_R^n \to V_R$ over $R$, or equivalently by $V_R(G_R^n) = V(F) \otimes_F \cO(G^n_R)$. The natural map of chain complexes of $R$-modules 
	\begin{equation}
		\label{eq: hochschild cohomology complex base change}
		\rC^{\bullet}(G,V) \otimes_F R \to \rC^{\bullet}(G_R, V_R)
	\end{equation}
	is an isomorphism since in degree $n$ it is given by 
	\[ (V(F) \otimes_F \cO(G^n)) \otimes_F R \lto V(F) \otimes_F \cO(G_R^n). \]
	Taking cohomology in~\eqref{eq: hochschild cohomology complex base change} and using that $R$ is flat over the field~$F$, we get an induced isomorphism
	\[ \rH^n(G, V) \otimes_F R \cong \rH^n(G_R, V_R) \quad \text{ for all $n \geq 0$}. \]
	In particular, choosing $n = 1$, the vanishing of $\rH^1(G, V)$ implies the vanishing of $\rH^1(G_R, V_R)$, so we are reduced to the case $R = F$.
	
	\emph{Step 4: The abelian case over $F$.}
	Let $V$ be a vector group with $G$-action. We have to show $\rH^1(G, V) = 0$. The cohomology group parametrises extensions of $G$-representations
	\[ 0 \lto V(F) \lto E \lto F  \lto 0, \]
	where $F$ is the trivial representation. Since $G$ is linearly reductive, any such extension $E$ is a direct sum of irreducible representations, which implies that the extension $E$ splits.
\end{proof}

\begin{proof}[Proof of \Cref{thm: existence of invariant point}]
	Let $P$ be a $\Pi$-torsor over~$R$ in $\MT(\cO_{K,S}, \bQ)$. The canonical realisation $P^{\omega}$ is then a $G_S^{\MT}$-equivariant $\Pi^{\omega}$-torsor over $R$. Restricting the action from $G_S^{\MT}$ to the subgroup $\bG_m$, the torsor $P^{\omega}$ defines an element of $\rH^1(\bG_{m,R}, \Pi^{\omega}_R)$. The pro-unipotent group $\Pi^{\omega}$ is finitely generated. This can be checked on the abelianisation and after tensoring with~$K$, i.e., for $\Pi^{\dR,\ab}$, where it follows from $\pi_1^{\dR}(X,0)^{\ab} = \rH^1_{\dR}(X/K)^\vee$ and
	\[ \rH^1_{\dR}(X/K) = K \frac{\rd t}{t} \oplus K \frac{\rd t}{1-t}. \]
	By \Cref{thm: cohomology of reductive group on unipotent group}, 
	the cohomology set $\rH^1(\bG_{m,R}, \Pi^{\omega}_R)$ is trivial, hence $P^{\omega}$ is trivial as a $\bG_m$-equivariant $\Pi^{\omega}$-torsor. This is equivalent to the existence of a $\bG_m$-invariant $R$-point.
	
	To show uniqueness, assume that $p_1$ and $p_2$ are two $\bG_m$-invariant $R$-points of $P^{\omega}$. Then the unique element $\gamma \in \Pi^{\omega}(R)$ such that $p_2 = p_1 \gamma$ is also $\bG_m$-invariant. But by \Cref{thm: fundamental group has trivial invariants} we have $(\Pi^{\omega})^{\bG_m} = 1$, thus $\gamma = 1$ and $p_1 = p_2$.
\end{proof}

\begin{defn}
	\label{def: canonical element}
	Let $P$ be a $\Pi$-torsor in $\MT(\cO_{K,S},\bQ)$ over a $\bQ$-algebra $R$. The unique $\bG_m$-invariant element $p^{\can}$ of $P^{\omega}(R)$, whose existence is guaranteed by \Cref{thm: existence of invariant point}, is called the \emph{canonical element}. 
\end{defn}

\subsection{Canonical paths on the thrice-punctured line}
\label{sec: canonical paths on the thrice-punctured line}

The main source of torsors under the motivic fundamental group are path torsors. As above, let $\pi_1^{\mot}(X,0) \twoheadrightarrow \Pi$ be a quotient of the motivic fundamental group of $X = \bP^1_{\cO_{K,S}} \smallsetminus \{0,1,\infty\}$ in $\MT(\cO_{K,S}, \bQ)$. For any $S$-integral base point $x$ of $X$ (see \Cref{def base point}), the motivic path space $\pi_1^{\mot}(X; 0,x)$ is a right $\pi_1^{\mot}(X,0)$-torsor in $\MT(\cO_{K,S},\bQ)$. We denote by $\pathtorsor{x}{0}$ its pushout along the quotient map $\pi_1^{\mot}(X,0) \twoheadrightarrow \Pi$. When $\Pi$ is the quotient of $\pi_1^{\mot}(X,0)$ by the normal subgroup~$N$, then we have
\[ \pathtorsor{x}{0} = \pi_1^{\mot}(X; 0,x)/N. \]
The $\pi_1^{\mot}(X,0)$-torsor structure of $\pi_1^{\mot}(X; 0,x)$ descends to a $\Pi$-torsor structure on~$\pathtorsor{x}{0}$. By \Cref{thm: existence of invariant point}, the canonical realisation $\pathtorsor[\omega]{x}{0} \coloneqq \omega(\pathtorsor{x}{0})$ has a unique $\bG_m$-invariant point.

\begin{defn}
	\label{def: canonical path}
	The unique $\bG_m$-invariant point
	\[ \pth[\can]{x}{0} \in \pathtorsor[\omega]{x}{0}(\bQ)^{\bG_m}, \]
	is called the \emph{canonical path}.
\end{defn}

\begin{rem}
	\label{rem:path-torsor-not-trivial}
The canonical path, though being $\bG_m$-invariant, is generally not $G_S^{\MT}$-invariant. Thus, it does not provide a trivialisation of the motivic path torsor.
\end{rem}

\subsection{Canonical cocycles}
\label{sec:canonical-cocycles}

By making use of canonical $\bG_m$-invariant points, we will prove another description of the Selmer functor in terms of $\bG_m$-equivariant cocycles. Recall from \Cref{sec:MT Galois group} that $G_S^{\MT}$ is a semidirect product $G_S^{\MT} = U_S^{\MT} \rtimes \bG_m$, so that $\bG_m$ acts on $U_S^{\MT}$ by conjugation. For a $\bQ$-algebra $R$, a cocycle $c\colon (U_S^{\MT})_R \to \Pi^{\omega}_R$ is $\bG_m$-equivariant if it satisfies 
\[ c(\lambda u \lambda^{-1}) = \lambda(c(u)) \quad \text{for all } \lambda \in \bG_{m,R} \text{ and } u \in (U_S^{\MT})_R \]
(valued in $R$-algebras). Denote by $\rZ^1((U_S^{\MT})_R, \Pi^{\omega}_R)^{\bG_m}$ the pointed set of $\bG_m$-equivariant cocycles, and by 
\[ \underline{\rZ}^1(U_S^{\MT}, \Pi^{\omega})^{\bG_m} \colon \; R \mapsto \rZ^1((U_S^{\MT})_R, \Pi^{\omega}_R)^{\bG_m} \]
the corresponding functor on $\bQ$-algebras. 

If $P$ is a $\Pi$-torsor in $\MT(\cO_{K,S},\bQ)$ over $R$, there is a unique $\bG_m$-invariant point $p^{\can} \in P^{\omega}(R)^{\bG_m}$ by \Cref{thm: existence of invariant point}. It defines a cocycle $c^{\can}_P \colon G_{S,R}^{\MT} \to \Pi^{\omega}_R$ by the rule
\[ g(p^{\can}) = p^{\can} \cdot c^{\can}_P(g) \quad \text{for $g \in G_S^{\MT}$}. \]
We call $c^{\can}_P$ the \emph{canonical cocycle}.

\begin{thm}
	\label{thm: equivariant cocycles description of Selmer scheme}
	The map $P \mapsto c_P^{\can}\vert_{U_S^{\MT}}$ defines an isomorphism of functors on $\bQ$-algebras
	\[ \Sel_{S,\Pi}^{\mot}(X) \cong \underline{\rZ}^1(U_S^{\MT}, \Pi^{\omega})^{\bG_m}. \]
\end{thm}

\begin{proof}
	Since the canonical element $p^{\can} \in P^{\omega}(R)$ is $\bG_m$-invariant, the canonical cocycle $c^{\can}_P$ is trivial on the subgroup $\bG_m \leq G_S^{\MT}$. In other words, $c^{\can}_P$ is contained in the kernel of the restriction map $\underline{\rZ}^1(G_S^{\MT}, \Pi^{\omega}) \to \underline{\rZ}^1(\bG_m, \Pi^{\omega})$. For any cocycle $c$ in this kernel, the restriction to $U_S^{\MT}$ is $\bG_m$-equivariant. Indeed, by using the cocycle property twice, for $\lambda \in \bG_m$ and $u \in U_S^{\MT}$ we have
	\[ c(\lambda u \lambda^{-1}) = c(\lambda) \cdot \lambda(c(u) \cdot u(c(\lambda^{-1}))) = 1 \cdot \lambda(c(u) \cdot u(1)) = \lambda(c(u)). \]
	We therefore have a diagram of functors of pointed sets on $\bQ$-algebras as follows:
	\[
	\begin{tikzcd}
		& \ker\Bigl( \underline{\rZ}^1(G_S^{\MT}, \Pi^{\omega}) \to \underline{\rZ}^1(\bG_m, \Pi^{\omega})  \Bigr) \dar \rar["\res"] & \underline{\rZ}^1(U_S^{\MT}, \Pi^{\omega})^{\bG_m} \\
		\Sel_{S,\Pi}^{\mot}(X) \arrow[ur, dashed, "c^{\can}"] \arrow[r, "\sim"] & \underline{\rH}^1(G_S^{\MT}, \Pi^{\omega}) & 
	\end{tikzcd}
	\]
	The commutativity of the triangle is clear from the construction in \Cref{sec:Sel-via-group-cohomology} of the bottom isomorphism. We claim that also the other maps in the diagram are isomorphisms. For the injectivity of the vertical map, assume that $c_1$ and $c_2$ are cocycles $G_{S,R}^{\MT} \to \Pi^{\omega}_R$ which are both trivial on $\bG_m$ and which are cohomologous. Then there exists $p \in \Pi^{\omega}(R)$ such that $c_2(g) = p^{-1} c_1(g) g(p)$ for all $g \in G_{S,R}^{\MT}$. This holds in particular for $g$ in the subgroup $\bG_m$ where both cocycles are trivial, which implies that $p$ is $\bG_m$-invariant. But we have $(\Pi^{\omega})^{\bG_m} = 1$ by \Cref{thm: fundamental group has trivial invariants}, hence $p = 1$ and $c_1 = c_2$. This shows the injectivity of the vertical map. The map is also split surjective via the canonical cocycle map, hence all maps in the triangle are isomorphisms.
	
	For the injectivity of the restriction map, suppose that $c\colon G_{S,R}^{\MT} \to \Pi^{\omega}_R$ is a cocycle which is trivial on $\bG_m$. Any element of $G_{S,R}^{\MT}$ (valued in an $R$-algebra~$R'$) can be written uniquely in the form $u \lambda$ with $u \in U_S^{\MT}(R')$ and $\lambda \in \bG_m(R')$. We have
	\[ c(u \lambda) = c(u) \cdot u(c(\lambda)) = c(u) \cdot u(1) = c(u), \]
	hence $c$ is determined by its restriction to $(U_S^{\MT})_R$.
	
	For the surjectivity of the restriction map, let $d\colon U_{S,R}^{\MT} \to \Pi^{\omega}_R$ be any $\bG_m$-equivariant cocycle. Extend this to a map $c\colon G_{S,R}^{\MT} \to \Pi^{\omega}_R$ by the rule
	\[ c(u \lambda) \coloneqq d(u) \]
	for $u \in (U_S^{\MT})_R$ and $\lambda \in \bG_{m,R}$. Then $c$ is a cocycle since for $g_i = u_i \lambda_i$ ($i=1,2$) we have
	\begin{align*}
		c(g_1 g_2) &= c(u_1 \lambda_1 u_2 \lambda_1^{-1} \cdot \lambda_1 \lambda_2) = d(u_1 \lambda_1 u_2 \lambda_1^{-1}) & \text{(by definition of $c$)} \\
		&= d(u_1) \cdot u_1(d(\lambda_1 u_2 \lambda_1^{-1})) & \text{(cocycle property)} \\
		&= d(u_1) \cdot u_1(\lambda_1(d(u_2))) & \text{($\bG_m$-equivariance)}\\
		&= c(u_1 \lambda_1) \cdot (u_1 \lambda_1)(c(u_2 \lambda_2)) & \text{(by definition of $c$)}\\
		&= c(g_1)\cdot g_1(c(g_2)).
	\end{align*}
	We have $c\vert_{(U_S^{\MT})_R} = d$ by construction, so the restriction map is surjective and hence an isomorphism.
\end{proof}

\section{The motivic Selmer scheme via Lie algebra cocycles}

The theory of pro-unipotent groups is equivalent to the theory of pro-nilpotent Lie algebras, and the latter are sometimes more convenient to work with. We show how cocycles of algebraic groups can be translated into cocycles of Lie algebras. This leads to an alternative description of the Selmer functor $\Sel^{\mot}_{S,\Pi}(X)$, which will be used to give a new proof of the representability by a (pro-)affine space in \Cref{sec:representability}.

We work over a field $F$ of characteristic~$0$.

\begin{defn}
	Let $\fu$ be a Lie algebra over~$F$. A \emph{derivation} of $\fu$ is a linear endomorphism $D$ which satisfies
	\[ D[Y,Z] = [DY, Z] + [Y, DZ] \quad \text{for } Y,Z \in \fu. \]
	The vector space of derivations of $\fu$ is denoted $\Der(\fu)$. It is itself a Lie algebra with the commutator bracket $[D_1, D_2] = D_1D_2 - D_2 D_1$. 
	An \emph{action} of a Lie algebra $\fg$ on $\fu$ is given by a Lie algebra homomorphism 
	\[ 	\phi\colon \fg \to \Der(\fu). \]
\end{defn}

The space of derivations $\Der(\fu)$ can be identified with the Lie algebra of the group $\Aut(\fu)$ of Lie algebra automorphisms of~$\fu$. 

\begin{example}
	\label{ex: Lie action from group action}
	If an algebraic group $G$ (with Lie algebra $\fg$) acts on another algebraic group $U$ (with Lie algebra $\fu$), then differentiating the homomorphism
	\[ G \to \Aut(U) \to \Aut(\fu) \]
	results in a Lie homomorphism $\fg \to \Lie \Aut(\fu) = \Der(\fu)$, i.e.\ an action by $\fg$ on $\fu$.
\end{example}

\begin{defn}
	\label{def: Lie algebra cocycles}
	Let $\fg$ be a Lie algebra acting on another Lie algebra $\fu$ via $\fg \to \Der(\fu)$, $X \mapsto \phi_X$. A \emph{1-cocycle} is a linear map $C\colon \fg \to \fu$ satisfying
	\[ C[X_1,X_2] = [CX_1, CX_2] + \phi_{X_1}(CX_2) - \phi_{X_2}(CX_1) \quad \text{for } X_1,X_2 \in \fg. \]
	The set of 1-cocycles is denoted by
	\[ \rZ^1(\fg, \fu). \]
\end{defn}

\begin{rem}
	\label{rem: cocycle to abelian Lie algebra}
	If the Lie algebra $\fu$ is abelian, then the $\fg$-action is simply a representation on the vector space $\fu$, and 1-cocycles are linear maps $C\colon \fg \to \fu$ satisfying $C[X_1,X_2] = \phi_{X_1}(CX_2) - \phi_{X_2}(CX_1)$ for $X_1,X_2 \in \fg$. This agrees with the definition of 1-cocycles in the Chevalley--Eilenberg complex which is used to define Lie algebra cohomology with coefficients in a representation \cite{chevalley-eilenberg}.
\end{rem}

1-cocycles of Lie algebras have an interpretation in terms of splittings of semidirect product extensions of Lie algebras:

\begin{defn}[{\cite[Ch.~I, §1.8]{bourbaki:lie_algebras}}]
	\label{def: semidirect product of Lie algebras}
	Let $\fg$ be a Lie algebra acting on $\fu$ via $X \mapsto \phi_X$. The \emph{semidirect product} $\fu \rtimes \fg$ is the Lie algebra whose underlying vector space is the direct sum $\fu \oplus \fg$ and whose Lie bracket is given by
	\[ [(Y_1,X_1), (Y_2,X_2)] = ([Y_1,Y_2] + \phi_{X_1}(Y_2) - \phi_{X_2}(Y_1), [X_1,X_2]) \]
	for $X_1, X_2 \in \fg$, $Y_1,Y_2 \in \fu$. The semidirect product sits in an extension
	\[ 0 \to \fu \to \fu \rtimes \fg \to \fg \to 0, \]
	which is canonically split via the inclusion $\fg \to \fu \rtimes \fg$.
\end{defn}

It is easy to verify that a linear map $C\colon \fg \to \fu$ is a 1-cocycle if and only if the map $X \mapsto (CX, X)$ defines a splitting of the Lie algebra homomorphism $\fu \rtimes \fg \to \fg$. Hence we have the following:

\begin{lemma}
	\label{lem: Lie algebra cocycles}
	The set of 1-cocycles $\rZ^1(\fg, \fu)$ is in bijection with splittings of the semidirect product extension $\fu \rtimes \fg$. \qed
\end{lemma}

\begin{rem}
	\label{rem: trivial Lie action}
	If the action of $\fg$ on $\fu$ is trivial, i.e., $\fg \to \Der(\fu)$ is the zero map, then 1-cocycles are simply homomorphisms of Lie algebras, and the semidirect product $\fu \rtimes \fg$ agrees with the direct product.
\end{rem}

\begin{lemma}
	\label{thm: Lie cocycle from group cocycle}
	Let $G$ be an algebraic group acting on another algebraic group $U$. Let $\fg \to \Der(\fu)$, $X \mapsto \phi_X$ be the induced action on Lie algebras (see \Cref{ex: Lie action from group action}).  If $c\colon G \to U$ is a cocycle of algebraic groups, then the derivative of $c$ at the identity is a cocycle of Lie algebras $C\colon \fg \to \fu$. 
\end{lemma}

\begin{proof}
	Let $U \rtimes G$ be the semidirect product of $U$ and $G$ with respect to the given action. This is again an algebraic group, and its Lie algebra is isomorphic to $\fu \rtimes \fg$. It is easy to verify that 1-cocycles $G \to U$ are equivalent to homomorphic sections of the projection $U \rtimes G \to G$. Let $\tilde c\colon G \to U \rtimes G$, $g \mapsto (c(g),g)$ be the section corresponding to the cocycle $c$. It induces a homomorphic section $\tilde C\colon \fg \to \fu \rtimes \fg$ on Lie algebras. By \Cref{lem: Lie algebra cocycles}, the composition with the projection $\fu \rtimes \fg \to \fu$ is a Lie algebra cocycle $C\colon \fg \to \fu$. Since the projection $\fu \rtimes \fg \to \fu$, albeit not a Lie algebra homomorphism, still is the derivative of the projection $U \rtimes G \to U$ at the identity, it follows that $C$ is indeed the derivative of $c$ at the identity.
\end{proof}

\Cref{thm: Lie cocycle from group cocycle} yields a map
\begin{equation}
	\label{eq: group cocycle to Lie cocycle}
	\rZ^1(G,U) \to \rZ^1(\fg,\fu)
\end{equation}
from algebraic group cocycles to Lie algebra cocycles given by taking the derivative at the identity. If $G$ and $U$ are both unipotent groups over $\bQ$, then so is the semidirect product $U \rtimes G$. Since homomorphisms of unipotent groups are equivalent to homomorphisms between their Lie algebras, the map~\eqref{eq: group cocycle to Lie cocycle} is a bijection in this case. This generalises to pro-unipotent groups and to cocycle functors in a straightforward way, so we conclude:

\begin{lemma}
	\label{lem: cocycle equivalence}
	Let $G/F$ be a pro-unipotent group with Lie algebra $\fg$ acting on another pro-unipotent group $U/F$ with Lie algebra $\fu$. Define the functor $\underline{\rZ}^1(\fg, \fu)$ on $F$-algebras by
	\[ R \mapsto \rZ^1(\fg_R, \fu_R). \]
	Then the map~\eqref{eq: group cocycle to Lie cocycle} defines an isomorphism
	\[ \underline{\rZ}^1(G,U) \cong \underline{\rZ}^1(\fg,\fu). \]
\end{lemma}

\begin{rem}
	\label{rem: Lie cocycle and logarithm}
	Cocycles, in contrast with homomorphisms, are in general not compatible with the logarithm map, i.e., if $c\colon G \to U$ is a cocycle and $C\colon \fg \to \fu$ is its derivative at the identity, then in general $\log(c(g)) \neq C(\log(g))$ for $g \in G$. 
\end{rem}

Applying \Cref{lem: cocycle equivalence} to the Selmer functor of the thrice-punctured line, we obtain the following:

\begin{prop}
	\label{thm: Selmer scheme via Lie cocycles}
	There is a canonical isomorphism
	\[\Sel_{S,\Pi}^{\mot}(X) \cong \underline{\rZ}^1(\Lie(U_S^{\MT}), \Lie(\Pi^{\omega}))^{\bG_m}. \]
	In other words, points of the Selmer scheme are equivalent to graded Lie cocycles. \qed
\end{prop}

\section{The Selmer scheme as an affine space}

We use the description of the Selmer functor in terms of Lie cocycles to show that it is representable by an affine space. 

\subsection{Generators for the unipotent mixed Tate Galois group}
\label{sec:generators of unipotent MT Galois group}

We start by discussing free generators of $\Lie(U_S^{\MT})$. The $\bG_m$-action on $U_S^{\MT}$ coming from the semidirect product structure of~$G_S^{\MT}$ yields a product grading on the Lie algebra of $U_S^{\MT}$, supported in negative degrees:
\[ \Lie(U_S^{\MT}) = \prod_{n=1}^\infty \Lie(U_S^{\MT})_{-n}. \]
We say that $\Lie(U_S^{\MT})$ is graded by \emph{half-weight}. The abelianisation $\Lie(U_S^{\MT})^{\ab}$ inherits a product grading and the abelianisation map is graded. If $\Sigma_{n}$ is a chosen lift of a basis of $\Lie(U_S^{\MT})^{\ab}_{-n}$ under the surjection
\[ \Lie(U_S^{\MT})_{-n} \twoheadrightarrow \Lie(U_S^{\MT})^{\ab}_{-n}, \]
then the set $\Sigma = \coprod_{n=1}^\infty \Sigma_{n}$ is a system of free homogeneous generators of $\Lie(U_S^{\MT})$ as a pro-nilpotent Lie algebra. We also refer to $\Sigma$ as a set of free generators for $U_S^{\MT}$. Note that the choice of $\Sigma$ is non-canonical.

The spaces $\Lie(U_S^{\MT})^{\ab}_{-n}$ are finite-dimensional and their dimensions are known. Namely, there are canonical isomorphisms~\cite[Proposition~A.15]{deligne-goncharov}
\[ \Lie(U_S^{\MT})^{\ab}_{-n} \cong \Ext^1_{\MT(\cO_{K,S},\bQ)}(\bQ(0), \bQ(n))^\vee \quad \text{for $n \geq 1$}, \]
the Ext groups are given by rational algebraic $K$-groups \cite[§1.7]{deligne-goncharov},
\[ \Ext^1_{\MT(\cO_{K,S},\bQ)}(\bQ(0), \bQ(n)) = K_{2n-1}(\cO_{K,S})_{\bQ} = \begin{cases}
	\cO_{K,S}^\times \otimes \bQ, & \text{if $n=1$},\\
	K_{2n-1}(K) \otimes \bQ &  \text{if $n\geq 2$},
\end{cases} \]
and their dimensions have been calculated by Borel~\cite{borel} as
\begin{equation}
	\label{eq:K-group-dimensions}
	\dim_\bQ K_{2n-1}(\cO_{K,S})_{\bQ} = \begin{cases}
	\#S + r_1 + r_2 -1, & \text{if $n=1$},\\
	r_2, & \text{if $n \geq 2$ even},\\
	r_1 + r_2, & \text{if $n \geq 3$ odd},
\end{cases}
\end{equation}
where $r_1$ and $r_2$ are the number of real and complex places of $K$, respectively. So by choosing a basis of $\Lie(U_S^{\MT})^{\ab}$ and lifting it to the full Lie algebra, we obtain a system $\Sigma = \coprod_{n=1}^{\infty} \Sigma_{n}$ of free homogeneous generators, and the number~$\#\Sigma_{n}$ of generators in degree~$-n$ is given by \eqref{eq:K-group-dimensions}.
In §\ref{sec:generators-of-lie-algebra} we address the question of how to construct such generators explicitly.

\subsection{Representability by an affine space}
\label{sec:representability}

We obtain a proof of the representability of the Selmer functor by an affine space:

\begin{thm}
	\label{thm: pro-representability via Lie cocycles}
	Let $\Sigma = \coprod_{n=1}^\infty \Sigma_{n}$ be a free homogeneous generating system of $\Lie(U_S^{\MT})$. Then there is an isomorphism of functors on $\bQ$-algebras
	\[ \Sel_{S,\Pi}^{\mot}(X) \cong \prod_{n=1}^\infty \Lie(\Pi^{\omega})_{-n}^{\Sigma_{n}}. \]
	In particular, $\Sel_{S,\Pi}^{\mot}(X)$ is a (possibly infinite-dimensional) affine space.
\end{thm}

Here, $\Lie(\Pi^{\omega})_{-n}^{\Sigma_{n}}$ denotes the $\bQ$-vector space of maps $\Sigma_{n} \to \Lie(\Pi^{\omega})_{-n}$, and a $\bQ$-vector space~$V$ is viewed as a functor on $\bQ$-algebras via $R \mapsto V \otimes_{\bQ} R$.

\begin{proof}
	Let $R$ be a $\bQ$-algebra and consider an $R$-valued point of $\Sel_{S,\Pi}^{\mot}(X)$, corresponding to a graded Lie algebra cocycle $C\colon \Lie (U_S^{\MT})_R \to \Lie(\Pi^{\omega})_R$ by \Cref{thm: Selmer scheme via Lie cocycles}. Since cocycles are equivalent to homomorphic sections of the semidirect product extension (\Cref{lem: Lie algebra cocycles}), the universal property of free pro-nilpotent Lie algebras applies to cocycles as well, so that $C$ corresponds uniquely to a map $\Sigma \to \Lie(\Pi^{\omega}) \otimes R$ on the free generators. The generators in degree~$-n$ are mapped into degree~$-n$, so the graded cocycle $C$ is equivalent to a collection of maps $\Sigma_{n} \to \Lie(\Pi^{\omega})_{-n} \otimes R$ for $n \geq 1$.
\end{proof}

\begin{cor}
	\label{cor: representability}
	If the quotient $\pi_1^{\mot}(X,0) \twoheadrightarrow \Pi$ is finite-dimensional, then the Selmer scheme $\Sel_{S,\Pi}^{\mot}(X)$ is an affine space of dimension
	\[ \dim \Sel_{S,\Pi}^{\mot}(X) = \sum_{n=1}^\infty \dim_{\bQ} K_{2n-1}(\cO_{K,S})_{\bQ} \cdot \dim_{\bQ} \Lie(\Pi^{\omega})_{-n},\]
	where the dimension of $K_{2n-1}(\cO_{K,S})_{\bQ}$ is given by~\eqref{eq:K-group-dimensions}.
\end{cor}

\begin{proof}
	If $\Pi$ is finite-dimensional, so is $\Lie(\Pi^{\omega})$, hence we have $\Lie(\Pi^{\omega})_{-n} = 0$ for $n \gg 0$. Then $\Sel_{S,\Pi}^{\mot}(X) \cong \prod_{n=1}^\infty \Lie(\Pi^{\omega})_{-n}^{\Sigma_{n}}$ is a finite product of finite-dimensional affine spaces, hence an affine space. The dimension formula follows from $\#\Sigma_{n} = \dim_{\bQ} K_{2n-1}(\cO_{K,S})_{\bQ}$.
\end{proof}

\subsection{Coordinates on the Selmer scheme}
\label{sec:coordinates-on-selmer-scheme}

Let $\pi_1^{\mot}(X,0) \twoheadrightarrow \Pi$ be an arbitrary quotient of the motivic fundamental group. \Cref{thm: pro-representability via Lie cocycles} shows that $\Sel_{S,\Pi}^{\mot}(X)$ is isomorphic to $\bA^N$ for some $N \in \bZ_{\geq 0} \cup \{\infty\}$. We construct a set of functions on the Selmer scheme via matrix coefficients which realise this isomorphism. 

\begin{defn}
	\label{def:selmer-scheme-functions}
	Let $n \in \bN$. Given $\sigma \in \Lie(U_S^{\MT})_{-n}$ and $f\in \Lie(\Pi^{\omega})_{-n}^{\vee}$ define
	\[ \Phi^{\sigma}_{f}\colon \Sel_{S,\Pi}^{\mot}(X) \to \bA^1 \]
	as the function sending a graded cocycle $c\colon \Lie(U_S^{\MT})_R \to \Lie(\Pi)_R$, defined over some $\bQ$-algebra~$R$, to
	\[ \Phi^{\sigma}_{f}(c) \coloneqq f(c(\sigma)) \in R. \]
	We refer to $\Phi^{\sigma}_{f}$ as a \emph{Selmer function}.
\end{defn}

\begin{thm}
	\label{coordinates-on-selmer-scheme}
	Let $\Sigma = \coprod_{n=1}^\infty \Sigma_{n}$ be a free homogeneous generating system of $\Lie(U_S^{\MT})$, and for each $n \geq 1$ let $f_{n,1},\ldots,f_{n,d_n}$ be a basis of $\Lie(\Pi^{\omega})_{-n}^{\vee}$. Then we have an isomorphism
	\[ \Sel_{S,\Pi}^{\mot}(X) \cong \Spec \bQ[(\Phi^{\sigma}_{f_{n,i}})_{n,\sigma,i}] \]
	with $n \in \bN$, $\sigma \in \Sigma_n$, and $1 \leq i \leq d_n$.
\end{thm}

\begin{proof}
	This follows immediately from \Cref{thm: pro-representability via Lie cocycles}.
\end{proof}

\Cref{coordinates-on-selmer-scheme} provides a recipe for constructing a set of coordinates on $\Sel_{S,\Pi}^{\mot}(X)$ for arbitrary fundamental group quotients. Note that for finite-dimensional~$\Pi$, only finitely many generators of $U_S^{\MT}$ need to be constructed: if $\Lie(\Pi^{\omega})_{-n} = 0$ for $n > N$, then only $\Sigma_1, \ldots, \Sigma_N$ are needed in \Cref{coordinates-on-selmer-scheme}. 
In order to choose the bases $\{f_{n,i}\}$ of $\Lie(\Pi^{\omega})_{-n}^{\vee}$, one possibility is to take a basis $\{b_{n,i}\}$ of $\Lie(\Pi^{\omega})_{-n}$ and define the $f_{n,i} \coloneqq b_{n,i}^\vee$ as the dual basis.
In the next section we describe such bases for several natural choices of~$\Pi$.
In §\ref{sec:coefficient-extraction-bases-for-lie-algebra} we discuss another construction of bases of $\Lie(\Pi^{\omega})_{-n}$ using coefficient extraction functionals.

\section{Dimension formulas}
\label{sec:dimension-formulas}

In this section we compute the dimension of the Selmer scheme for several natural quotients of the fundamental group:
\begin{itemize}
	\item the full fundamental group;
	\item the polylogarithmic quotient;
	\item the metabelian quotient.
\end{itemize}
In each case, we further truncate at some finite nilpotency depth~$n$ in order to obtain a finite-dimensional quotient~$\pi_1^{\mot}(X,0) \twoheadrightarrow \Pi$ and compute the dimension of the associated finite-dimensional Selmer scheme $\Sel_{S,\Pi}^{\mot}(X)$. Using \Cref{cor: representability}, this amounts to computing the dimensions of the weight-graded pieces $\Lie(\Pi^{\omega})_{-n}$ of the Lie algebra $\Lie(\Pi^{\omega})$.

\subsection{Structure of the fundamental group}
\label{sec:structure-of-fundamental-group}

The canonical realisation $\pi_1^{\omega}(X,0)$ of the motivic fundamental group of $X = \bP^1 \smallsetminus \{0,1,\infty\}$ is a free pro-unipotent group on two generators. Equivalently, its Lie algebra $\Lie(\pi_1^{\omega}(X,0))$ is free pro-nilpotent on two generators. The Lie algebra generators $e_0$ and $e_1$ can be viewed as infinitesimal loops around~$0$ and~$1$, respectively. They are canonical and can be constructed as follows.

Let $T_0$ be the tangent space of $X$ at~$0$, and let $T_0^* = T_0 \smallsetminus \{0\} \cong \bG_m$. The motivic fundamental group of $\bG_m$ with base point~$1$ is canonically isomorphic to~$\bQ(1)$, viewed as a vector group of rank one, so its canonical realisation is isomorphic to $\bG_a$. The theory of tangential base points yields a \emph{local monodromy map} in $\MT(\cO_{K,S},\bQ)$ \cite[§5.4]{deligne-goncharov}
\begin{equation}
	\label{eq:local-monodromy-0}
	\mu_0\colon \bQ(1) = \pi_1^{\mot}(T_0^*, 1) \to \pi_1^{\mot}(X,0).
\end{equation}
Taking canonical realisations and passing to the induced map on Lie algebras we get
\[ \rd\mu_0^{\omega}\colon \bQ = \Lie(\pi_1^{\omega}(T_0^*, 1)) \to \Lie(\pi_1^{\omega}(X,0)) \]
and $e_0$ is defined as the image of $1 \in \bQ$ under this map. 

We have a similar local monodromy map at 1
\[ \mu_1\colon \bQ(1) = \pi_1^{\mot}(T_1^*,-1) \to \pi_1^{\mot}(X,1) \]
It takes values in the motivic fundamental group of~$X$ at the tangential base point~$1$, which is shorthand for the tangent vector $-\vec{1}_1$ at~$1$, corresponding to~$-1$ under $T_1 \cong \bA^1$. We get an element $e_1' \in \Lie(\pi_1^{\omega}(X,1))$, defined as the image of~$1 \in \bQ$ under $\rd \mu_1^{\omega}$. Let $\gamma \in \pi_1^{\omega}(X;0,1)(\bQ)$ denote the canonical path from~$0$ to~$1$ on~$X$ (\Cref{def: canonical path}). Conjugation by~$\gamma$ defines an isomorphism
\[ \psi_\gamma \colon \pi_1^{\omega}(X,0) \cong \pi_1^{\omega}(X,1), \]
which induces an isomorphism on Lie algebras
\[ \rd \psi_{\gamma}\colon \Lie(\pi_1^{\omega}(X,0)) \cong \Lie(\pi_1^{\omega}(X,1)). \]
We can then define $e_1 \coloneqq \rd \psi_{\gamma}^{-1}(e_1') \in \Lie(\pi_1^{\omega}(X,0))$.

\begin{prop}
	\label{free-generators}
	The Lie algebra $\Lie(\pi_1^{\omega}(X,0))$ is free pro-nilpotent on the two generators $e_0$ and $e_1$.
\end{prop}

\begin{proof}
	This can be checked after tensoring with~$K$, i.e.\ on the de Rham realisation, where it follows from \cite[Proposition~12.10]{deligne:droite-projective}.
\end{proof}

Note that the generators $e_0$ and $e_1$ of $\Lie(\pi_1^{\omega}(X,0))$ are homogeneous of half-weight~$-1$ since they come from $\bQ(1)$ and the canonical path~$\gamma$ is $\bG_m$-invariant. 

\subsection{The full depth~$N$ quotient}
\label{sec:full-depth-n}

Let $\Pi = \pi_1^{\mot}(X,0)$ be the full fundamental group and let $\Pi \twoheadrightarrow \Pi_N$ be its depth~$N$ quotient for some $N \geq 1$, i.e., the maximal quotient of nilpotency depth~$\leq N$. Thus, $\Pi_1 = \Pi/[\Pi,\Pi] = \Pi^{\ab}$ is the abelianisation, $\Pi_2 =\Pi/[\Pi,[\Pi,\Pi]]$ is the maximal quotient which is a central extension of $\Pi^{\ab}$ etc. On the level of Lie algebras, $\Lie(\Pi_N)$ is the quotient of $\Lie(\Pi)$ by all nested commutators of length $> N$. This makes sense motivically inside $\MT(\cO_{K,S},\bQ)$. Since $\Lie(\Pi^\omega_N)$ is generated in half-weight~$-1$, the descending central series~$\fz^{\bullet}$ and the weight filtration $W_{\bullet}$ on $\Lie(\Pi_N)$ are related by
\begin{equation}
	\label{eq:weight-vs-lower-central-series}
	\fz^i \Lie(\Pi_N) = W_{-2i} \Lie(\Pi_N).
\end{equation}

\begin{lemma}
	\label{dim-of-graded-lie}
	The dimension of $\Lie(\Pi_N^{\omega})_{-n}$ is given by
	\[ \dim_{\bQ} \Lie(\Pi_N^{\omega})_{-n} = 
	\begin{cases}
		M(2,n), & \textnormal{if $n \leq N$},\\
		0, & \textnormal{if $n > N$,}
	\end{cases} \]
	where $M(\alpha,n) = \frac1{n} \sum_{d\mid n} \mu(n/d) \alpha^d$ denotes the $n$-th Moreau necklace polynomial.
\end{lemma}

\begin{proof}
	By~\eqref{eq:weight-vs-lower-central-series}, we have
	\[ \Lie(\Pi_N^{\omega})_{-n} = \gr^W_{-2n} \Lie(\Pi_N^\omega) = \gr_{\fz}^n \Lie(\Pi_N^\omega). \]
	When $n \leq N$, the latter is the $n$-th graded piece with respect to the lower central series of a free Lie algebra on two generators, and it is well-known that its dimension is given by a necklace polynomial, see for example \cite[Theorem~6]{reutenauer03}.
\end{proof}

The first few values of $M(2,n)$, $n \geq 1$, are given by
\[ 2, \; 1, \; 2, \; 3, \; 6, \; 9, \; 18, \; 30, \; 56, \; 99,\ldots \]

Denote by $\Sel_{S,N}^{\mot}(X) \coloneqq \Sel_{S,\Pi_N}^{\mot}(X)$ the motivic Selmer scheme for the depth $N$ quotient of the fundamental group.

\begin{prop}
	\label{full-quotient-dimension}
	 The motivic Selmer scheme $\Sel_{S,N}^{\mot}(X)$ for the depth~$N$ fundamental group is an affine space of dimension
	\[ \dim \Sel_{S,N}^{\mot}(X) = 2\#S - 2 + r_2 \sum_{\substack{2\leq n \leq N\\ \textnormal{even}}} M(2,n) + (r_1 + r_2) \sum_{\substack{1\leq n \leq N\\ \textnormal{odd}}} M(2,n),\]
	where $r_1$ and $r_2$ denote the number of real and complex places of~$K$, respectively.
\end{prop}

\begin{proof}
	This follows from \Cref{cor: representability} together with the formula~\eqref{eq:K-group-dimensions} for the dimensions of the $K$-groups $K_{2n-1}(\cO_{K,S})_{\bQ}$ and \Cref{dim-of-graded-lie}.
\end{proof}

For example, when $K$ is totally real of degree~$d$ over~$\bQ$, we have
\[ \dim \Sel_{S,N}^{\mot}(X) = 2\#S - 2 + d \sum_{\substack{1\leq n \leq N\\ \textnormal{odd}}} M(2,n). \]
When $K$ is totally imaginary of degree~$d$, we have
\[ \dim \Sel_{S,N}^{\mot}(X) = 2\#S - 2 + \frac{d}{2} \sum_{1\leq n \leq N} M(2,n). \]
The dimensions of $\Sel_{S,N}^{\mot}(X)$ for $S = \emptyset$ and small~$N$ over different base fields are listed in \Cref{table:full-depth-n-dimensions}. The dimensions for nonempty~$S$ can be obtained by adding $2\#S$. 

\begin{table}
		\begin{tabular}{|c||cccccccccc|}
			\hline
			depth~$N$ & 1 & 2 & 3 & 4 & 5 & 6 & 7 & 8 & 9 & 10 \\
			 \hline
			$K = \bQ$ & 0 & 0 & 2 & 2 & 8 & 8 & 26 & 26 & 82 & 82 \\
			$K$ real quadratic & 2 & 2 & 6 & 6 & 18 & 18 & 54 & 54 & 166 & 166 \\
			$K$ imaginary quadratic & 0 & 1 & 3 & 6 & 12 & 21 & 39 & 69 & 125 & 224 \\
			\hline
		\end{tabular}
		\caption{Dimensions of Selmer schemes for the full depth~$N$ fundamental group. The table entries are~$\dim \Sel^{\mot}_{\emptyset,N}(X)$ over different base fields and for $1 \leq N \leq 10$. 
		}
		\label{table:full-depth-n-dimensions}
\end{table}

\subsubsection*{The Lyndon basis}
\label{sec:lyndon-basis}

Recall from §\ref{sec:coordinates-on-selmer-scheme} that the construction of coordinates on the Selmer scheme $\Sel_{S,N}^{\mot}(X)$ depends on choosing a basis $\{f_{n,i}\}$ of $\Lie(\Pi_N^{\omega})_{-n}^{\vee}$ for $1 \leq n \leq N$, and one way to achieve this is by taking the dual basis of a given basis of $\Lie(\Pi_N^{\omega})_{-n}$. Such bases can be constructed from so-called \emph{Hall sets} \cite[§4]{reutenauer-book}. A special case of this produces the \emph{Lyndon basis} \cite[§5.1]{reutenauer-book}, which we describe briefly. 

A word $w$ on the letters $e_0$ and $e_1$ is called a \emph{Lyndon word} if it is nonempty and lexicographically strictly smaller (with respect to the ordering $e_0  < e_1$) than all its nontrivial cyclic rotations. The Lyndon words of length~$\leq 4$ are
\begin{equation}
	\label{eq:lyndon-words}
	e_0, \quad e_1, \quad e_0 e_1, \quad e_0^2 e_1, \quad  e_0 e_1^2, \quad e_0^3 e_1, \quad e_0^2 e_1^2, \quad e_0 e_1^3.
\end{equation}
Given a Lyndon word~$w$ of length~$n$, there is an associated element $\lambda(w) \in \Lie(\pi_1^\omega(X,0))_{-n}$ obtained by inserting Lie brackets in a recursive manner: for a single letter one sets $\lambda(e_0) \coloneqq e_0$ and $\lambda(e_1) \coloneqq e_1$; for $|w| \geq 2$ one defines $\lambda(w) \coloneqq [\lambda(u),\lambda(v)]$, where $w = uv$ is the \emph{standard factorisation}, i.e.\ the factorisation of~$w$ into Lyndon words $u$ and $v$ with $v$ as long as possible. The element $\lambda(w)$ is called the \emph{standard bracketing} of~$w$. For example:
\[ \lambda(e_0 e_1^2) = [\lambda(e_0 e_1), \lambda(e_1)] = [[e_0,e_1],e_1].\]
The Lyndon words with their standard bracketing form the \emph{Lyndon basis}:

\begin{prop}
	\label{lyndon-basis}
	Let $n \geq 1$. The elements $\lambda(w)$ with $w$ a Lyndon word on $\{e_0 < e_1\}$ of length~$n$ form a basis of $\Lie(\pi_1^{\omega}(X,0))_{-n}$. 
\end{prop}

\begin{proof}
	See for example \cite[§5.1]{reutenauer-book}.
\end{proof}

Thus, from the Lyndon words~\eqref{eq:lyndon-words} one obtains the following bases of $\Lie(\Pi^{\omega})_{-n}$ for $1 \leq n \leq 4$:

\begin{center}
\begin{tabular}{|c|c|c|}
	\hline
	$n$ & basis of $\Lie(\Pi^{\omega})_{-n}$ & dimension $M(2,n)$ \\
	\hline
	$1$ & $e_0, \quad e_1$ & $2$ \\
	$2$ & $[e_0, e_1]$ & $1$  \\
	$3$ & $[e_0, [e_0,e_1]], \quad  [[e_0, e_1],e_1]$ & $2$  \\
	$4$ & $[e_0, [e_0, [e_0,e_1]]], \quad [e_0, [[e_0, e_1],e_1]], \quad [[[e_0, e_1],e_1],e_1]$ & $3$ \\
	\hline
\end{tabular}
\end{center}

\subsection{The polylogarithmic quotient}
\label{sec:polylogarithmic}

One fundamental group quotient which is particularly convenient to work with in the Chabauty--Kim method is the so-called \emph{polylogarithmic quotient}. It can be defined as follows: the inclusion $X = \bP^1 \smallsetminus \{0,1,\infty\} \hookrightarrow \bG_m$ induces a homomorphism of motivic fundamental groups $\pi_1^{\mot}(X,0) \to \bQ(1)$. Let $N_1$ be its kernel. The polylogarithmic quotient $\pi_1^{\mot}(X,0) \twoheadrightarrow \Pi_{\PL}$ is defined as
\[ \Pi_{\PL} \coloneqq \pi_1^{\mot}(X,0)/[N_1,N_1]. \]
On the level of Lie algebras, the kernel of $\Lie(\pi_1^\omega(X,0)) \to \Lie(\bG_a) = \bQ$ is the Lie ideal generated by~$e_1$, so $\Lie(\Pi_{\PL}^\omega)$ can be described as the free pro-nilpotent Lie algebra on generators~$\{e_0,e_1\}$ modulo all nested commutators of $e_1$-degree $> 1$. Consequently, a basis of $\Lie(\Pi_{\PL}^\omega)$ is given by $\{e_0,e_1\}$ in half-weight~$-1$ and $\ad(e_0)^{n-1} e_1$ in half-weight~$-n$ for $n \geq 2$.

For $N \geq 1$, denote by $\Pi_{\PL,N}$ the maximal quotient of $\Pi_{\PL}$ of nilpotency depth~$\leq N$, and denote by $\Sel_{S,\PL,N}^{\mot}(X)$ the associated Selmer scheme.

\begin{lemma}
	\label{polylog-lie-dimensions}
	The dimension of $\Lie(\Pi_{\PL,N}^\omega)_{-n}$ is given by
	\[ \dim_{\bQ} \Lie(\Pi_{\PL,N}^{\omega})_{-n} = 
	\begin{cases}
		2 & \textnormal{if $n = 1$},\\
		1 & \textnormal{if $2 \leq n \leq N$},\\
		0, & \textnormal{if $n > N$.}
	\end{cases} \]
\end{lemma}

\begin{proof}
	This is clear from the description of the bases of the weight-graded pieces of $\Lie(\Pi_{\PL}^{\omega})$.
\end{proof}

\begin{prop}
	\label{polylog-dimension}
	The depth~$N$ polylogarithmic Selmer scheme $\Sel^{\mot}_{S,\PL,N}(X)$ is an affine space of dimension
	\[ \dim \Sel^{\mot}_{S,\PL,N}(X) = r_1 \left\lfloor \frac{N+3}{2} \right\rfloor + r_2(N+1) + 2\#S - 2,\]
	where $r_1$ and $r_2$ denote the number of real and complex places of~$K$, respectively.
\end{prop}

\begin{proof}
	Combining \Cref{cor: representability} with the formula~\eqref{eq:K-group-dimensions} for the dimensions of the $K$-groups $K_{2n-1}(\cO_{K,S})_{\bQ}$ and \Cref{polylog-lie-dimensions}, we obtain
	\begin{align*}
		\dim \Sel^{\mot}_{S,\PL,N}(X) &= (r_1 + r_2 - 1 + \#S)\cdot 2 + r_2 \sum_{\substack{2\leq n \leq N\\ \textnormal{even}}} 1 + (r_1+r_2)\sum_{\substack{3\leq n \leq N\\ \textnormal{odd}}} 1 \\
		&= 2(r_1 + r_2 - 1 + \#S) + r_1 \#\{3\leq n \leq N \textnormal{ odd}\} + r_2 \#\{2 \leq n \leq N\} \\
		&= 2(r_1 + r_2 - 1 + \#S) + r_1 \left\lfloor \frac{N-1}{2} \right\rfloor + r_2(N-1) \\
		&= r_1 \left\lfloor \frac{N+3}{2} \right\rfloor + r_2(N+1) + 2\#S - 2. \qedhere
	\end{align*}
\end{proof}

For example, when $K$ is totally real of degree~$d$ over~$\bQ$, we have
\[ \dim \Sel^{\mot}_{S,\PL,N}(X) = d \left\lfloor \frac{N+3}{2} \right\rfloor + 2\#S - 2. \]
When $K$ is totally imaginary of degree~$d$, we have
\[ \dim \Sel^{\mot}_{S,\PL,N}(X) = \frac{d}{2}(N+1) + 2\#S - 2. \]
The dimensions of $\Sel^{\mot}_{S,\PL,N}(X)$ for $S = \emptyset$ and small~$N$ over different base fields are listed in \Cref{table:polylog-dimensions}. The dimensions for nonempty~$S$ can be obtained by adding $2\#S$. Note that the dimensions for the polylogarithmic quotient are considerably smaller than for the full fundamental group.

\begin{table}
	\begin{tabular}{|c||cccccccccc|}
		\hline
		depth~$N$ & 1 & 2 & 3 & 4 & 5 & 6 & 7 & 8 & 9 & 10 \\
		\hline
		$K = \bQ$ & 0 & 0 & 1 & 1 & 2 & 2 & 3 & 3 & 4 & 4 \\
		$K$ real quadratic & 2 & 2 & 4 & 4 & 6 & 6 & 8 & 8 & 10 & 10 \\
		$K$ imaginary quadratic & 0 & 1 & 2 & 3 & 4 & 5 & 6 & 7 & 8 & 9 \\
		\hline
	\end{tabular}
	\caption{Dimensions of polylogarithmic Selmer schemes. The table entries are $\dim \Sel^{\mot}_{\emptyset,\PL,N}(X)$ over different base fields and for $1 \leq N \leq 10$.}
	\label{table:polylog-dimensions}
\end{table}

\subsection{The metabelian quotient}
\label{sec:metabelian}

We now consider the metabelian quotient $\Pi_{\met}$ of the fundamental group. It is defined as the quotient by the second derived subgroup (commutators of commutators), so if $\Pi = \pi_1^{\mot}(X,0)$ denotes the full fundamental group then
\[ \Pi_{\met} \coloneqq \Pi/[[\Pi,\Pi], [\Pi,\Pi]]. \]
Note that this definition makes sense inside $\MT(\cO_{K,S},\bQ)$. The metabelian quotient is still infinite-dimensional, so we consider its depth $N$ quotients $\Pi_{\met,N}$ for $N \geq 1$.

\begin{lemma}
	\label{metabelian-basis}
	A basis of $\Lie(\Pi_{\met,N}^{\omega})$ is given by $e_0$, $e_1$, and $\ad(e_0)^{k-1} \ad(e_1)^{l-1} [e_0,e_1]$ for $k,l\geq 1$ with $k+l\leq N$.
\end{lemma}

\begin{proof}
	For any nested commutator $Y \in \Lie(\Pi^\omega)$ of length~$\geq 2$ the Jacobi identity implies
	\[ [e_1, [e_0, Y]] - [e_0, [e_1,Y]] = [[e_0, e_1], Y] = 0 \text{  in $\Lie(\Pi_{\met}^{\omega})$}, \]
	so occurrences of $e_0$ and $e_1$ can be permuted. Thus, in the metabelian Lie algebra, any nested commutator of length~$\geq 2$ is equal up to sign to $\ad(e_0)^{k-1} \ad(e_1)^{l-1}[e_0,e_1]$ for some $k,l \geq 1$, showing that the listed elements generate $\Lie(\Pi^{\omega}_{\met,N})$ as a $\bQ$-vector space. To see that they form a basis one can use \cite[Theorem~5.7]{reutenauer-book}. The theorem describes a Hall set~$H = H_0 \cup H_1 \cup \ldots$ whose associated Hall basis is compatible with the derived series in the sense that the elements coming from $\bigcup_{k \geq p} H_k$ form a basis of the $p$-th derived Lie ideal in the (uncompleted) free Lie algebra for $p \geq 1$. In particular, the elements coming from $H_0$ and $H_1$ form a basis of the metabelian quotient. These are $e_0$, $e_1$, as well as all elements of the form
	\begin{equation*}
		\label{eq:hall-basis-element}
		[[[[\cdots [[[[e_0,e_1],e_1],\ldots,e_1],e_0],e_0],\ldots],e_0]
	\end{equation*}
	with $e_0$ appearing $k \geq 1$ times and $e_1$ appearing $l \geq 1$ times. 
	Using $[Y,e_i] = \ad(-e_i)Y$, this can be rewritten as
	\[ \ad(-e_0)^{k-1} \ad(-e_1)^{l-1}[e_0,e_1], \]
	which agrees with the claimed basis element $\ad(e_0)^{k-1} \ad(e_1)^{l-1}[e_0,e_1]$ up to sign. Finally, we obtain a basis of the depth~$N$ quotient $\Lie(\Pi_{\met,N}^{\omega})$ by ignoring all basis elements with $k+l > N$.
\end{proof}

\begin{lemma}
	\label{metabelian-lie-dimensions}
	The dimension of $\Lie(\Pi_{\met,N}^\omega)_{-n}$ is given by
	\[ \dim_{\bQ} \Lie(\Pi_{\met,N}^{\omega})_{-n} = 
	\begin{cases}
		2 & \textnormal{if $n = 1$},\\
		n-1 & \textnormal{if $2 \leq n \leq N$},\\
		0, & \textnormal{if $n > N$.}
	\end{cases} \]
\end{lemma}

\begin{proof}
	By \Cref{metabelian-basis}, the basis elements of $\Lie(\Pi_{\met,N}^\omega)$ in half-weight $-1$ are $e_0$ and $e_1$; the basis elements in half-weight $-n$, for $2 \leq n \leq N$, are $\ad(e_0)^{k-1} \ad(e_1)^{n-k-1} [e_0,e_1]$ with $k=1,\ldots,n-1$.
\end{proof}

Denote the Selmer scheme for the depth $N$ metabelian quotient by $\Sel_{S,\met,N}^{\mot}(X)$.

\begin{prop}
	\label{metabelian-dimension}
	The metabelian Selmer scheme $\Sel^{\mot}_{S,\met,N}(X)$ is an affine space of dimension
	\[ \dim \Sel_{S,\met,N}^{\mot}(X) = r_1 \left\lfloor \frac{N-1}{2} \right\rfloor \left\lfloor \frac{N+1}{2} \right\rfloor + r_2 \frac{N(N-1)}{2} + 2(r_1 + r_2 -1 +\#S). \]
\end{prop}

\begin{proof}
	Combining \Cref{cor: representability} with the formula~\eqref{eq:K-group-dimensions} for the dimensions of the $K$-groups $K_{2n-1}(\cO_{K,S})_{\bQ}$ and \Cref{metabelian-lie-dimensions} yields
	\begin{align*}
		\dim \Sel_{S,\met,N}^{\mot}(X) &= 2(r_1 + r_2 - 1 + \#S) + r_2 \sum_{\substack{2\leq n \leq N\\ \textnormal{even}}} (n-1) + (r_1+r_2)\sum_{\substack{3\leq n \leq N\\ \textnormal{odd}}} (n-1) \\
		&= 2(r_1 + r_2 - 1 + \#S) + r_1\sum_{\substack{3\leq n \leq N\\ \textnormal{odd}}} (n-1) + r_2 \sum_{2 \leq n \leq N} (n-1) \\
		&= 2(r_1 + r_2 - 1 + \#S) + r_1 \sum_{1 \leq i \leq \lfloor (N-1)/2 \rfloor} 2i + r_2 \frac{N(N-1)}{2} \\
		&= 2(r_1 + r_2 - 1 + \#S) + r_1 \left\lfloor \frac{N-1}{2} \right\rfloor \left(\left\lfloor \frac{N-1}{2} \right\rfloor + 1 \right) + r_2 \frac{N(N-1)}{2} \\
		&= r_1 \left\lfloor \frac{N-1}{2} \right\rfloor \left\lfloor \frac{N+1}{2} \right\rfloor + r_2 \frac{N(N-1)}{2} + 2(r_1 + r_2 -1 +\#S). \qedhere
	\end{align*}
\end{proof}

For example, when $K$ is totally real of degree~$d$ over~$\bQ$, we have
\[ \dim \Sel_{S,\met,N}^{\mot}(X) = d \left\lfloor \frac{N-1}{2} \right\rfloor \left\lfloor \frac{N+1}{2} \right\rfloor + 2(d -1 +\#S). \]
When $K$ is totally imaginary of degree~$d$, we have
\[ \dim \Sel_{S,\met,N}^{\mot}(X) =  \frac{d}{2}\cdot \frac{N(N-1)}{2} + d - 2 + 2\#S. \]
The dimensions of $\Sel_{S,\met,N}^{\mot}(X)$ for $S = \emptyset$ and small~$N$ over different base fields are listed in \Cref{table:metabelian-dimensions}. The dimensions for nonempty~$S$ can be obtained by adding $2\#S$.

\begin{table}
	\begin{tabular}{|c||cccccccccc|}
		\hline
		depth~$N$ & 1 & 2 & 3 & 4 & 5 & 6 & 7 & 8 & 9 & 10 \\
		\hline
		$K = \bQ$ & 0 & 0 & 2 & 2 & 6 & 6 & 12 & 12 & 20 & 20 \\
		$K$ real quadratic & 2 & 2 & 6 & 6 & 14 & 14 & 26 & 26 & 42 & 42 \\
		$K$ imaginary quadratic & 0 & 1 & 3 & 6 & 10 & 15 & 21 & 28 & 36 & 45 \\
		\hline
	\end{tabular}
	\caption{Dimensions of metabelian Selmer schemes. The table entries are $\dim \Sel^{\mot}_{\emptyset,\met,N}(X)$ over different base fields and for $1 \leq N \leq 10$.}
	\label{table:metabelian-dimensions}
\end{table}

\section{The motivic Selmer scheme via Hopf algebra cocycles}
\label{sec:hopf-algebra-cocycles}

When doing explicit calculations in Chabauty--Kim theory, it can be useful to work with yet another type of object besides algebraic groups and Lie algebras: complete Hopf algebras. In this section we work over any field $F$ of characteristic zero.

\subsection{Profinite-dimensional vector spaces}
\label{sec:profinite-dimensional-vector-spaces}

A \emph{profinite-dimensional vector space} $V$ over~$F$ is a topological vector space which is isomorphic to an inverse limit of finite-dimensional discrete vector spaces, the limit being endowed with the inverse limit topology. We denote by $\ProFinVect_F$ the category of profinite-dimensional vector spaces with continuous linear maps. This category is isomorphic to the pro-category of finite-dimensional vector spaces. A family $\{v_i\}_{i \in I}$ of elements in a profinite-dimensional vector space is said to form a \emph{basis} if the map 
\[ F^I \to V, \quad (a_i)_i \mapsto \sum_{i \in I} a_i v_i \]
is a topological isomorphism for the product topology on $F^I$. Here, the sum is interpreted as the limit of partial sums over finite subsets of~$I$. Every basis is linearly independent and spans a dense subspace. 

If $V$ is an arbitrary discrete vector space, not necessarily finite-dimensional, its linear dual $V^{\vee}$ naturally carries the structure of a profinite-dimensional vector space via the isomorphism
\[ V^{\vee} \cong \varprojlim_i V_i^{\vee}, \]
with $V_i$ running through all finite-dimensional subspaces of~$V$. This induces a duality
\begin{equation}
	\label{eq:vector-space-duality}
	\Vect_F^{\op} \cong \ProFinVect_F
\end{equation}
between vector spaces (and linear maps) and profinite-dimensional vector spaces (and continuous linear maps) given by $V \leftrightarrow V^{\vee}$. Here, $V^{\vee}$ denotes the ordinary dual space of~$V$ when $V$ is a discrete vector space, and it denotes the \emph{continuous} dual space when $V$ is a profinite-dimensional vector space. Since $\Vect_F$ is an abelian category and the axioms for an abelian category are self-dual, $\ProFinVect_F$ is an abelian category as well. It is equipped with the \emph{completed tensor product}, i.e.\ for $V, W \in \ProFinVect_F$ we define
\[ V \otimes W \coloneqq V \hat\otimes W \coloneqq \varprojlim_{i,j}\, (V_i \otimes W_j) \quad \text{where } V = \varprojlim_i V_i, \; W = \varprojlim_j W_j \]
with $V_i$ and $W_j$ running through the finite-dimensional (topological) quotients of $V$ and $W$, respectively. The duality~\eqref{eq:vector-space-duality} is monoidal with respect to the ordinary tensor product on $\Vect_F$ and the completed tensor product on $\ProFinVect_F$.

\subsection{The complete Hopf algebra associated to an algebraic group}
\label{sec:cocommutative-hopf-algebra}

\begin{defn}
	\label{def:complete-hopf-algebra}
	A \emph{complete Hopf algebra} over~$F$ is a Hopf algebra object in $\ProFinVect_F$.
\end{defn}

If $G$ is an affine group scheme over~$F$, the ring of functions $\cO(G)$ is a commutative Hopf algebra with coproduct $\cO(G) \to \cO(G) \otimes \cO(G)$ induced by the group operation $G \times G \to G$. By the duality~\eqref{eq:vector-space-duality}, its dual $\cO(G)^\vee$ is a cocommutative complete Hopf algebra in the sense of \Cref{def:complete-hopf-algebra}.

\begin{defn}
	\label{def:cocommutative-hopf-algebra}
	The \emph{complete Hopf algebra} (or \emph{convolution algebra}) associated to an affine group scheme~$G$ is
	\[ \cH(G) \coloneqq \cO(G)^\vee. \]
\end{defn}

The Hopf algebra structure on $\cH(G)$ is induced from the one on $\cO(G)$ by dualising. Concretely:
\begin{enumerate}
	\item the multiplication $\nabla\colon \cO(G) \otimes \cO(G) \to \cO(G)$ dualises to the \emph{coproduct}
	\[ \Delta \colon \cH(G) \to \cH(G) \otimes \cH(G); \]
	\item the unit map $\eta\colon F \to \cO(G)$ dualises to the \emph{counit} $\varepsilon\colon \cH(G) \to F$;
	\item the coproduct $\Delta\colon \cO(G) \to \cO(G) \otimes \cO(G)$ coming from the group operation $G \times G \to G$ dualises to the \emph{product}
	\[ \nabla\colon \cH(G) \otimes \cH(G) \to \cH(G); \]
	\item the counit $\varepsilon\colon \cO(G) \to F$ (evaluating functions on~$G$ at the identity) dualises to the \emph{unit} $\eta\colon F \to \cH(G)$;
	\item the antipode $\iota^{\sharp}\colon \cO(G) \to \cO(G)$ coming from the inversion map $\iota\colon G \to G$ dualises to the antipode $S\colon \cH(G) \to \cH(G)$.
\end{enumerate}

Note that since $\cH(G)$ has the structure of a profinite-dimensional vector space, $\cH(G) \otimes \cH(G)$ denotes the \emph{completed} tensor product, which can be larger than the tensor product of the underlying abstract vector spaces. 

The association $G \mapsto \cH(G)$ is covariantly functorial: a homomorphism of affine group schemes $\phi\colon G \to G'$ is given by a Hopf algebra homomorphism $\phi^\sharp\colon \cO(G') \to \cO(G)$, and its dual $(\phi^\sharp)^{\vee}$ is a homomorphism of complete Hopf algebras
\[ \cH(\phi)\colon \cH(G) \to \cH(G'). \]
We get an equivalence of categories
\begin{align*}
	\left\{\text{affine group schemes} \right\} &\simeq \left\{\text{cocommutative complete Hopf algebras} \right\},\\
	G &\mapsto \cH(G).
\end{align*}
The quasi-inverse is given by $H \mapsto \Spec(H^{\vee})$. 

\begin{example}
	\label{ex:finite group}
	If $G$ is a finite group, viewed as a finite group scheme over~$F$, its ring of functions is $\cO(G) = F^G$ and its associated complete Hopf algebra $\cH(G)$ is the group algebra $F[G]$. The isomorphism $(F^G)^{\vee} \cong F[G]$ is given by $\mu \mapsto \sum_{g \in G} \mu(\delta_g) g$, where $\delta_g \in F^G$ is the Kronecker delta function $g' \mapsto \delta_{gg'}$.
\end{example}

\begin{rem}
	\label{tannaka-interpretation-of-hopf-algebra}
	The complete Hopf algebra $\cH(G)$ of an affine group scheme has a Tannakian interpretation as the algebra of endomorphisms of the forgetful fibre functor on the category of finite-dimensional representations of~$G$.
\end{rem}

\subsection{Group-like and Lie-like elements}
\label{sec:grouplike-and-lielike}

An affine group scheme $G$ over~$F$ can be recovered from its associated complete Hopf algebra via $G = \Spec \cH(G)^\vee$. The group of $F$-valued points $G(F)$ as well as the Lie algebra $\Lie(G)$ can be found inside $\cH(G)$ as the set of group-like and Lie-like elements, respectively.

\begin{defn}
	\label{def:group-like}
	Let $H$ be a cocommutative complete Hopf algebra. An element~$X$ of~$H$ is called \emph{group-like} if $\varepsilon(X) = 1$ and $\Delta X= X \otimes X$. The set of group-like elements in~$H$ is denoted $H^{\gp}$.
\end{defn}

The set of group-like elements in $H$ is closed under multiplication and forms a group. For an $F$-valued point $g \in G(F)$, evaluation at $g$ defines an algebra homomorphism $g^\sharp\colon \cO(G) \to F$, hence an element of $\cH(G)$.

\begin{lemma}
	\label{lem:grouplike}
	This defines an isomorphism of groups $G(F) \cong \cH(G)^{\gp}$.
\end{lemma}

\begin{proof}
	The group-like elements of $\cH(G)$ are linear maps $X\colon \cO(G) \to F$ satisfying $X(1) = 1$ and $X(ab) = X(a)X(b)$, in other words $F$-algebra homomorphisms:
	\[ \cH(G)^{\gp} = \Hom_{F\mhyphen\mathrm{alg}}(\cO(G), F) = G(F). \]
	One checks that under this isomorphism, multiplication in $\cH(G)^{\gp}$ is the group operation on $G(F)$.
\end{proof}

\begin{defn}
	\label{def:lie-like}
	Let $H$ be a cocommutative complete Hopf algebra. An element~$X$ of~$H$ is called \emph{Lie-like} (or \emph{primitive}) if $\varepsilon(X) = 0$ and $\Delta X = X \otimes 1 + 1 \otimes X$. The set of Lie-like elements in~$H$ is denoted $H^{\Lie}$.
\end{defn}

The set of Lie-like elements in~$H$ is closed under the bracket $[a,b] \coloneqq ab - ba$ and thus forms a Lie algebra. For a pro-algebraic group~$G$, elements of the Lie algebra $\Lie(G)$ can be identified with derivations $D\colon \cO(G) \to F$, i.e., linear maps satisfying
\[ D(ab) = \varepsilon(a) D(b) + \varepsilon(b) D(a) \]
for $a,b \in \cO(G)$ \cite[Proposition~10.28]{milne:algebraic}. Hence, we have an inclusion $\Lie(G) \hookrightarrow \cH(G)^{\Lie}$.

\begin{lemma}
	\label{lem:lielike}
	This defines an isomorphism of Lie algebras $\Lie(G) \cong \cH(G)^{\Lie}$.
\end{lemma}

\begin{proof}
	For $X \in  \cH(G)$, unfolding the definition of the coproduct and counit, the property of being Lie-like is equivalent to $X\colon \cO(G) \to F$ being a derivation. One checks that under this isomorphism, the Lie bracket on $\cH(G)$ is the Lie bracket on $\Lie(G)$.
\end{proof}

If $G$ is a pro-unipotent group, it follows from \cite[A.6--A.8]{deligne-goncharov} that the inclusion $\Lie(G) \hookrightarrow \cH(G)$ realises $\cH(G)$ as the completed enveloping algebra of $\Lie(G)$.

\subsection{Hopf algebra cocycles}
\label{sec: Hopf cocycles}

Besides algebraic group cocycles and Lie algebra cocycles, we can also consider cocycles of Hopf algebras. 

\begin{defn}
	\label{def: Hopf algebra action}
	Let $H$ and $V$ be cocommutative complete Hopf algebras over~$F$. A \emph{(left) action} of $H$ on $V$ is a homomorphism of coalgebras in profinite-dimensional vector spaces
	\[ H \otimes V \to V, \quad X \otimes v \mapsto X.v \]
	satisfying
	\begin{enumerate}
		\makeatletter
		\item[(1a)] \def\@currentlabel{(1a)} \label{item:hopf-action-1}
		$1.v = v$ for $v \in V$; 
		\item[(1b)] \def\@currentlabel{(1b)} \label{item:hopf-action-associative}
		$(XY).v = X.(Y.v)$ for $X,Y \in H$, $v \in V$;
		\item[(2a)] \def\@currentlabel{(2a)} \label{item:hopf-action-on-1}
		$X.1 = \varepsilon(1)1$ for $X \in H$;
		\item[(2b)] \def\@currentlabel{(2b)} \label{item:hopf-action-on-product}
		$X.(vw) = \sum_{(X)} (X_{(1)}.v) (X_{(2)}.w)$ for $X \in H$, $v,w \in V$.
	\end{enumerate}
\end{defn}

Here, we are using Sweedler notation in condition~\ref{item:hopf-action-on-product}. If $H$ and $V$ are finite-dimensional, the right hand side of \ref{item:hopf-action-on-product} means $\sum_i (X_{(1)}^i.v)(X_{(2)}^i.w)$ when the coproduct of $X$ is written as $\Delta X = \sum X_{(1)}^i \otimes X_{(2)}^i$. In general, one has to pass to finite-dimensional quotients, or interpret the formulas in \Cref{def: Hopf algebra action} as shorthand for expressing the commutativity of certain diagrams in the category of profinite-dimensional vector spaces with the completed tensor product; for example, \ref{item:hopf-action-on-product} expresses the equality of two maps $H \otimes V \otimes V \to V$.

Suppose $G$ is a pro-algebraic group over $F$ which acts on another pro-algebraic group $U$. The action map $G \times U \to U$ corresponds to an algebra homomorphism $\cO(U) \to \cO(G) \otimes \cO(U)$. Taking the linear dual yields a homomorphism of coalgebras in profinite-dimensional vector spaces
\begin{equation*}
	\label{eq: Hopf action}
	\tag{$\ast$}
	\cH(G) \otimes \cH(U) \to \cH(U).
\end{equation*}

\begin{lemma}
	\label{hopf-action-from-group-action}
	The map \eqref{eq: Hopf action} defines a Hopf algebra action in the sense of \Cref{def: Hopf algebra action}.
\end{lemma}

\begin{proof}
	The conditions \ref{item:hopf-action-1} and \ref{item:hopf-action-associative} follow from the associativity of the group action ($1.u = u$ and $gh.u = g.(h.u)$), while conditions \ref{item:hopf-action-on-1} and \ref{item:hopf-action-on-product} follow from the fact that $G$ acts by group automorphisms ($g.1 = 1$ and $g.(u_1 u_2) = (g.u_1)(g.u_2)$). For example, the last formula, saying that the $G$-action respects the multiplication on~$U$, is expressed by the commutative diagram
	\[
	\begin{tikzpicture}[
		>=stealth,
		font=\small,
		x=\dimexpr\textwidth/20\relax,
		y=\dimexpr\textwidth/24\relax
		]
		\node (GU) at (0,2) {$G \times U$};
		\node (GUU) at (-8,0) {$G \times U \times U$};
		\node (U) at (8,0) {$U$};
		\node (GGUU) at (-6,-2) {$G \times G \times U \times U$};
		\node (GUGU) at (1,-2) {$G \times U \times G \times U$};
		\node (UU) at (6,-2) {$U \times U$};
		\draw[->] (GUU) -- node[above left] {$\id \times \operatorname{mult}$} (GU);
		\draw[->] (GU) -- node[above right]{$\operatorname{act}$} (U);
		\draw[->] (GUU) -- node[above right]{$\Delta \times \id \times \id$} (GGUU);
		\draw[->] (GGUU) -- node[above]{$\id \times \operatorname{swap} \times \id$} (GUGU);
		\draw[->] (GUGU) -- node[above]{$\operatorname{act} \times \operatorname{act}$} (UU);
		\draw[->] (UU) -- node[above left]{$\operatorname{mult}$} (U);
	\end{tikzpicture}
	\]
	Applying the ``ring of functions'' functor (which reverses the arrows) and taking linear duals (which reverses them back) results in the following commutative diagram of complete Hopf algebras with coalgebra homomorphisms:
\[
	\begin{tikzpicture}[
			>=stealth,
			font=\scriptsize,
			x=\dimexpr\textwidth/19\relax,
			y=\dimexpr\textwidth/20\relax
		]
		\node (GU) at (0,1.5) {$\cH(G) \otimes \cH(U)$};
		\node (GUU) at (-8,0) {$\cH(G) \otimes \cH(U) \otimes \cH(U)$};
		\node (U) at (8,0) {$\cH(U)$};
		\node (GGUU) at (-7,-2) {$\cH(G) \otimes \cH(G) \otimes \cH(U) \otimes \cH(U)$};
		\node (GUGU) at (1,-2) {$\cH(G) \otimes \cH(U) \otimes \cH(G) \otimes \cH(U)$};
		\node (UU) at (7,-2) {$\cH(U) \otimes \cH(U)$};
		\draw[->] (GUU) -- node[above left] {$\id \otimes \operatorname{mult}$} (GU);
		\draw[->] (GU) -- node[above right]{$\operatorname{act}$} (U);
		\draw[->] (GUU) -- node[right]{$\Delta \otimes \id \otimes \id$} (GGUU);
		\draw[->] (GGUU) -- node[above]{$\id \otimes \operatorname{swap} \otimes \id$} (GUGU);
		\draw[->] (GUGU) -- node[above]{$\operatorname{act} \otimes \operatorname{act}$} (UU);
		\draw[->] (UU) -- node[left]{$\operatorname{mult}$} (U);
	\end{tikzpicture}
\]
	Tracing the element $X \otimes v \otimes w$ (for $X \in \cH(G)$, $v,w \in \cH(U)$) along the top path we get
	\[ X \otimes v \otimes w \mapsto X \otimes vw \mapsto X.(vw). \]
	Tracing it along the bottom path, we get
	\begin{align*}
		X \otimes v \otimes w &\mapsto \sum_{(X)} X_{(1)} \otimes X_{(2)} \otimes v \otimes w \mapsto \sum_{(X)} X_{(1)} \otimes v \otimes X_{(2)} \otimes w \mapsto \sum_{(X)} (X_{(1)}.v) \otimes (X_{(2)}.w)\\
		&\mapsto \sum_{(X)} (X_{(1)}.v)(X_{(2)}.w).
	\end{align*}
	Thus, the commutativity of the diagram is equivalent to \ref{item:hopf-action-on-product}. The other conditions for a Hopf algebra action are derived similarly.
\end{proof}

\begin{remarks}
	\label{rem: Hopf action recovers group and Lie action}
	\leavevmode
	\begin{enumerate}
		\item If $G$ is an affine algebraic group over~$F$ then by~\ref{def: Hopf algebra action}\,\ref{item:hopf-action-on-1}--\ref{item:hopf-action-on-product}, an element $g \in G(F)$, viewed as a group-like element of $\cH(G)$ (see \Cref{lem:grouplike}), acts on $\cH(U)$ as an algebra automorphism. Restricting to group-like elements of $\cH(U)$ recovers the group action of $G(F)$ on $U(F)$. 
		
		\item Similarly, for an element of the Lie algebra $X \in \Lie(G)$, viewed as a Lie-like element of $\cH(G)$ (see \Cref{lem:lielike}), the conditions \ref{def: Hopf algebra action}\,\ref{item:hopf-action-on-1}--\ref{item:hopf-action-on-product} say that~$X$ acts on $\cH(U)$ as a derivation:
		\[ X.(vw) = (X.v)w + v(X.w). \]
		Restricting also to Lie-like elements of $\cH(U)$ recovers the Lie algebra action of $\Lie(G)$ on $\Lie(U)$ by derivations which was defined in \Cref{ex: Lie action from group action}.
	\end{enumerate}
\end{remarks}

\begin{defn}
	\label{def: Hopf cocycles}
	Let $H$ and $V$ be cocommutative complete Hopf algebras over $F$, and suppose $H$ acts on $V$ (\Cref{def: Hopf algebra action}). A \emph{Hopf algebra cocycle} is a coalgebra homomorphism $c\colon H \to V$ satisfying
	\begin{equation}
		\label{eq:hopf-algebra-cocycle}
		c(XY) = \sum_{(X)} c(X_{(1)}) \cdot X_{(2)}.c(Y) \quad \text{for $X,Y \in H$}.
	\end{equation}
\end{defn}

If $G$ is an affine group scheme acting on another affine group scheme $U$, then a cocycle $c\colon G \to U$ corresponds to an algebra homomorphism $c^{\sharp}\colon \cO(U) \to \cO(G)$, whose linear dual is a map of coalgebras 
\begin{equation}
	\label{eq:hopf-cocycle}
	c\coloneqq (c^{\sharp})^\vee\colon \cH(G) \to \cH(U).
\end{equation}

\begin{lemma}
	\label{hopf-cocycle-from-group-cocycle}
	The association $c \mapsto (c^{\sharp})^{\vee}$ defines a bijection between group scheme cocycles $G \to U$ and Hopf algebra cocycles $\cH(G) \to \cH(U)$.
\end{lemma}

\begin{proof}
	Let $c \colon G \to U$ be a group cocycle. We verify that the map \eqref{eq:hopf-cocycle} is a Hopf algebra cocycle. The cocycle condition $c(gh) = c(g)\cdot (g.c(h))$ is expressed by the following commutative diagram:
	\[
	\begin{tikzpicture}[
		>=stealth,
		font=\small,
		x=\dimexpr\textwidth/20\relax,
		y=\dimexpr\textwidth/24\relax
		]
		\node (G) at (0,2) {$G$};
		\node (GG) at (-8,0) {$G \times G$};
		\node (U) at (8,0) {$U$};
		\node (GGG) at (-6,-2) {$G \times G \times G$};
		\node (UGU) at (1,-2) {$U \times G \times U$};
		\node (UU) at (6,-2) {$U \times U$};
		\draw[->] (GG) -- node[above left] {$\operatorname{mult}$} (G);
		\draw[->] (G) -- node[above right]{$c$} (U);
		\draw[->] (GG) -- node[above right]{$\Delta \times \id$} (GGG);
		\draw[->] (GGG) -- node[above]{$c \times \id \times c$} (UGU);
		\draw[->] (UGU) -- node[above]{$\id \times \operatorname{act}$} (UU);
		\draw[->] (UU) -- node[above left]{$\operatorname{mult}$} (U);
	\end{tikzpicture}
	\]
	Applying the ``ring of functions'' functor and dualising yields the following commutative diagram of complete Hopf algebras and coalgebra homomorphisms:
	\[
	\begin{tikzpicture}[
		>=stealth,
		font=\small,
		x=\dimexpr\textwidth/19\relax,
		y=\dimexpr\textwidth/20\relax
		]
		\node (G) at (0,1.5) {$\cH(G)$};
		\node (GG) at (-8,0) {$\cH(G) \otimes \cH(G)$};
		\node (U) at (8,0) {$\cH(U)$};
		\node (GGG) at (-6.5,-2) {$\cH(G) \otimes \cH(G) \otimes \cH(G)$};
		\node (UGU) at (1,-2) {$\cH(U) \otimes \cH(G) \otimes \cH(U)$};
		\node (UU) at (6.5,-2) {$\cH(U) \otimes \cH(U)$};
		\draw[->] (GG) -- node[above left] {$\id \otimes \operatorname{mult}$} (G);
		\draw[->] (GU) -- node[above right]{$c$} (U);
		\draw[->] (GG) -- node[right=2pt]{$\Delta \otimes \id$} (GGG);
		\draw[->] (GGG) -- node[above]{$c \otimes \id \otimes\, c$} (UGU);
		\draw[->] (UGU) -- node[above]{$\id \otimes \operatorname{act}$} (UU);
		\draw[->] (UU) -- node[left=2pt]{$\operatorname{mult}$} (U);
	\end{tikzpicture}
	\]
	In element notation this is precisely the Hopf algebra cocycle condition~\eqref{eq:hopf-algebra-cocycle}. Since taking rings of functions and linear duals are equivalences of categories, the construction $c \mapsto (c^{\sharp})^{\vee}$ is a bijection.
\end{proof}

\begin{remarks}
	\label{rem: Hopf coycle recovers group and Lie cocycle}
	\leavevmode
	\begin{enumerate}
		\item Restricting \eqref{eq:hopf-cocycle} to group-like elements recovers the group cocycle $c\colon G(F) \to U(F)$.
		\item Restricting \eqref{eq:hopf-cocycle} to Lie-like elements recovers the Lie algebra cocycle $\Lie(G) \to \Lie(U)$ from \Cref{thm: Lie cocycle from group cocycle}.
	\end{enumerate}
\end{remarks}

This justifies using the same notation for the group cocycle $c\colon G \to U$ and the corresponding Hopf algebra cocycle $\cH(G) \to \cH(U)$. 

\subsection{The motivic Selmer scheme via Hopf cocycles}
\label{subsec:motivic-selmer-scheme-via-hopf-cocycles}

Returning to our specific situation, let $\pi_1(X,0) \twoheadrightarrow \Pi$ be a motivic quotient of the fundamental group of the thrice-punctured line at the tangential base point $0 = \vec{1}_0$, so that we have an action of the unipotent motivic Galois group $U_S^{\MT}$ on $\Pi^{\omega}$. By the previous discussion we have a Hopf algebra action
\[ \cH(U_S^{\MT}) \otimes \cH(\Pi^{\omega}) \to \cH(\Pi^{\omega}). \]
For a $\bQ$-algebra~$R$, denote by $\rZ^1(\cH(U_S^{\MT})_R, \cH(\Pi^{\omega})_R)^{\bG_m}$ the set of $\bG_m$-equivariant Hopf algebra cocycles, and denote by
\[ \underline{\rZ}^1(\cH(U_S^{\MT}), \cH(\Pi^{\omega}))^{\bG_m} \]
the corresponding functor on $\bQ$-algebras. A caveat is in order when applying the formalism of complete Hopf algebras from §\ref{sec:profinite-dimensional-vector-spaces}--\ref{sec: Hopf cocycles} over an algebra rather than a field. In this case the language of pro- and ind-objects is better suited than the language of topological modules and continuous linear maps. That is, over a $\bQ$-algebra~$R$, the duality \eqref{eq:vector-space-duality} generalises to
\[ \Ind(\Proj_f(R))^{\op} \cong \Pro(\Proj_f(R)), \quad M \mapsto M^{\vee} \coloneqq \Hom_R(M,R), \]
where $\Proj_f(R)$ denotes the category of finitely generated projective $R$-modules, and $\cH(U_S^{\MT})_R$ and $\cH(\Pi^{\omega})_R$ are viewed as Hopf algebra objects in $\Pro(\Proj_f(R))$ (with pro-completed tensor product). The notions of ``Hopf algebra action'' and ``Hopf algebra cocycle'' generalise in the obvious way.

By §\ref{sec: Hopf cocycles}, $\bG_m$-equivariant cocycles $U_S^{\MT} \to \Pi^{\omega}$ are equivalent to graded cocycles of complete Hopf algebras, so we find:

\begin{prop}
	\label{thm: Selmer scheme via Hopf cocycles}
	There is a canonical isomorphism
	\[\Sel_{S,\Pi}^{\mot}(X) \cong \underline{\rZ}^1(\cH(U_S^{\MT}), \cH(\Pi^{\omega}))^{\bG_m}. \]
\end{prop}

\section{Free pro-unipotent groups}
\label{sec:free-prounipotent-groups}

In this section we discuss the free pro-unipotent group on a set of generators $\Sigma$ and the structures associated with it. In the context of this paper, the discussion applies both to $\pi_1^{\omega}(X,0)$, which is free on two generators $\{e_0,e_1\}$ (see §\ref{sec:structure-of-fundamental-group}), and to $U_S^{\MT}$, which is free on infinitely many generators $\Sigma = \coprod_{n=1}^{\infty} \Sigma_{n}$ (see §\ref{sec:generators of unipotent MT Galois group}).

Let $\Sigma$ be any set and let $F$ be a field of characteristic zero. We denote by $\fn(\Sigma)$ the free pro-nilpotent Lie algebra on~$\Sigma$. This is by definition the inverse limit 
\[ \fn(\Sigma) = \varprojlim_{\Sigma'\subseteq \Sigma \text{ finite}} \varprojlim_n\; \fn(\Sigma')/\fz^{n+1} \fn(\Sigma') \]
of the finite-dimensional nilpotent Lie algebras $\fn(\Sigma')/\fz^{n+1} \fn(\Sigma')$, with $\fz^{\bullet}$ denoting the descending central series. This makes $\fn(\Sigma)$ a Lie algebra object in the category $\ProFinVect_F$ of profinite-dimensional vector spaces over~$F$ with the completed tensor product. We denote by $U(\Sigma)$ the corresponding pro-unipotent group and call $U(\Sigma)$ the \emph{free pro-unipotent group} on~$\Sigma$. Note that the free generators in~$\Sigma$ are elements of the Lie algebra $\fn(\Sigma)$ rather than elements of the group $U(\Sigma)(F)$. Denote by $\cH(\Sigma) = \cO(U(\Sigma))^{\vee}$ the complete Hopf algebra associated to $U(\Sigma)$ (\Cref{def:cocommutative-hopf-algebra}). It has a concrete description in terms of non-commutative power series.

\begin{prop}[{\cite[A.9]{deligne-goncharov}}]
	\label{free-hopf-algebra}
	The complete Hopf algebra $\cH(\Sigma)$ of the free pro-unipotent group~$U(\Sigma)$ is the algebra of non-commutative power series:
	\[ \cH(\Sigma) = F\llangle \Sigma \rrangle. \]
\end{prop}

Elements of $F\llangle \Sigma\rrangle$ have the form $\sum_w a_w \, w$ with $a_w \in F$ and $w$ running through the words over~$\Sigma$. Note that we allow infinite sums of bounded degree as elements of $F\llangle \Sigma \rrangle$, so for example $\sum_{\sigma \in \Sigma} \sigma \in F\llangle \Sigma \rrangle$ even when $\Sigma$ is infinite.
The (non-commutative) product on $F\llangle \Sigma\rrangle$ is given by concatenation of words. The coproduct is determined by 
\[ \Delta(\sigma) = 1 \otimes \sigma + \sigma \otimes 1 \]
for $\sigma \in \Sigma$, the counit $\varepsilon\colon \cH(\Sigma) \to F$ sends a power series to its constant coefficient, and the antipode is given on monomials by
\[ S(\sigma_1\cdots \sigma_n) = (-1)^n \sigma_n \cdots \sigma_1. \]

As discussed in §\ref{sec:grouplike-and-lielike}, elements of $U(\Sigma)(F)$ can be identified with group-like elements of $F\llangle\Sigma\rrangle$, and elements of $\fn(\Sigma)$ can be identified with Lie-like elements of $F\llangle \Sigma\rrangle$.
For any Lie-like element $X \in F\llangle \Sigma \rrangle$, since $\varepsilon(X) = 0$, the power series~$X$ has no constant term, so $\exp(X) = \sum X^n/n!$ is a well-defined element of $F\llangle \Sigma \rrangle$, which can be shown to be group-like. In this way, the exponential map $\exp\colon \fn(\Sigma) \to U(\Sigma)(F)$ can be identified with the map $F\llangle \Sigma\rrangle^{\Lie} \to F\llangle \Sigma \rrangle^{\gp}$ given by the exponential series. Similarly, the inverse map $\log\colon U(\Sigma)(F) \to \fn(\Sigma)$ can be identified with the map $F\llangle \Sigma \rrangle^{\gp} \to F\llangle \Sigma \rrangle^{\Lie}$ given by the logarithm series 
\[ \log(X) = \sum_{n=1}^\infty \frac{(-1)^{n+1}}{n} (X-1)^n. \]

The identification $\cH(U(\Sigma)) = F\llangle \Sigma \rrangle$ can also be used to define functions on~$U(\Sigma)$ and $\fn(\Sigma)$. Since $\cH(G)$ is defined as the dual of $\cO(G)$, functions on~$U$ are simply elements of the (topological) dual space of~$F\llangle \Sigma \rrangle$:
\begin{equation}
	\label{eq:functions-via-dual}
	\cO(U(\Sigma)) = F\llangle \Sigma \rrangle^\vee.
\end{equation}
A basis of this dual space is given by the coefficient extraction functionals:

\begin{defn}
	\label{def:coefficient-extraction-functional}
	Let $w$ be a word on the alphabet~$\Sigma$. The \emph{coefficient extraction functional} $f_w \in \cO(U(\Sigma))$ is the function $F\llangle\Sigma\rrangle \to F$ mapping a power series $\sum_u a_u u$ to its $w$-coefficient.
\end{defn}

Restricting $f_w$ to group-like elements recovers the associated function $U(\Sigma)(F) \to F$ on $F$-valued points. One can also restrict $f_w$ to Lie-like elements to obtain linear functionals on $\fn(\Sigma)$. The multiplication of functions in $\cO(U(\Sigma))$ is dual to the coproduct in $\cH(\Sigma) = F\llangle \Sigma \rrangle$. As a consequence, the product of two coefficient extraction functionals is given by the \emph{shuffle product}:
\begin{equation}
	\label{eq:shuffle-product}
	f_w f_{w'} = \sum_{\pi \in \Sh(|w|, |w'|)} f_{\pi(ww')}.
\end{equation}
Here, $\Sh(n,m) \subseteq S_{n+m}$ denotes the set of all permutations satisfying $\pi^{-1}(1) < \ldots < \pi^{-1}(n)$ and $\pi^{-1}(n+1) < \ldots < \pi^{-1}(n+m)$, and for $\pi \in S_{n+m}$ and a word $\tau_1\cdots \tau_{n+m}$ of length~$n+m$ on~$\Sigma$ we write $\pi(\tau_1\cdots \tau_{n+m}) \coloneqq \tau_{\pi(1)} \cdots \tau_{\pi(n+m)}$. Thus, $\cO(U(\Sigma))$ is naturally the \emph{shuffle algebra} on~$\Sigma$.

The property of a power series to be group-like or Lie-like can be expressed as a condition on its coefficients as follows:

\begin{lemma}
	\label{lem: group-like and primitive coefficients}
	Let $A = \sum_{w} a_w \cdot w$ be a power series in $F\llangle \Sigma\rrangle$.
	\begin{enumerate}[label=(\alph*)]
		\item \label{item: group-like coefficients}
		$A$ is group-like if and only if its constant term satisfies $a_{\emptyset} = 1$ and for all words $w_1, w_2$, the shuffle relation holds:
		\[ a_{w_1} a_{w_2} = \sum_{\pi \in \Sh(|w_1|, |w_2|)} a_{\pi(w_1w_2)}. \]
		Here the sum is over all shuffles of $w_1,w_2$.
		
		\item \label{item: primitive coefficients}
		$A$ is Lie-like if and only if $a_{\emptyset} = 0$ and for all non-empty words $w_1,w_2$, the sum over all shuffled coefficients is zero:
		\[ \sum_{\pi \in \Sh(|w_1|, |w_2|)} a_{\pi(w_1 w_2)} = 0. \]
	\end{enumerate}
\end{lemma}

\begin{proof}
	The counit on $F\llangle \Sigma \rrangle$ is given by extracting the constant coefficient. The coproduct is given by
	\[ \Delta w = \sum_{\substack{w_1,w_2,\pi\\ \pi(w_1w_2) = w}} w_1 \otimes w_2, \]
	the sum being over all ways of writing $w$ as a shuffle permutation of two words $w_1$ and $w_2$. This follows formally from the coproduct being multiplicative and given by $\Delta\sigma = 1 \otimes \sigma + \sigma \otimes 1$ on generators $\sigma \in \Sigma$. Writing down the group-like and Lie-like conditions and comparing coefficients of all $w_1 \otimes w_2$ yields the claimed formulas.
\end{proof}

\section{Functions on the fundamental group and its Lie algebra}
\label{sec:functions-on-fundamental-group}

\subsection{Coefficient extraction bases for the Lie algebra}
\label{sec:coefficient-extraction-bases-for-lie-algebra}

Recall from §\ref{sec:coordinates-on-selmer-scheme} that the construction of coordinates on the Selmer scheme $\Sel_{S,\Pi}^{\mot}(X)$ depends on choosing a basis of $\Lie(\Pi^{\omega})_{-n}^{\vee}$ for $n \geq 1$. One way of achieving this is by choosing a basis of $\Lie(\Pi^{\omega})_{-n}$ and taking its dual basis. For some choices of $\Pi$ we discussed specific bases of $\Lie(\Pi^{\omega})_{-n}$ in Section~\ref{sec:dimension-formulas}:
\begin{itemize}
	\item for $\Pi = \pi_1^{\mot}(X,0)$ the full fundamental group, the Lyndon basis of $\Lie(\Pi^{\omega})_{-n}$ consists of the standard bracketings $\lambda(w)$ with $w$ ranging over the Lyndon words of length~$n$ over the alphabet $\{e_0 < e_1\}$;
	\item for $\Pi = \Pi_{\PL}$ the polylogarithmic quotient, a basis of $\Lie(\Pi_{\PL}^{\omega})_{-n}$ is given by $e_0$ and $e_1$ if $n = 1$, and by $\ad(e_0)^{n-1}e_1$ if $n > 1$;
	\item for $\Pi = \Pi_{\met}$ the metabelian quotient, a basis of $\Lie(\Pi_{\met}^{\omega})_{-n}$ is given by $e_0$ and $e_1$ if $n = 1$, and by $\ad(e_0)^{k-1} \ad(e_1)^{n-k-1} [e_0,e_1]$ with $k=1,\ldots,n-1$, if $n > 1$.
\end{itemize}

We now describe a second way to construct bases of the dual spaces $\Lie(\Pi^{\omega})_{-n}^{\vee}$ using the coefficient extraction functionals from \Cref{def:coefficient-extraction-functional}.

As discussed in §\ref{sec:structure-of-fundamental-group}, the fundamental group $\pi_1^{\omega}(X,0)$ is free pro-unipotent on two canonical generators $e_0, e_1$. Thus, by~§\ref{sec:free-prounipotent-groups}, the complete Hopf algebra associated to $\pi_1^{\omega}(X,0)$ is the non-commutative power series algebra $\bQ\llangle e_0,e_1 \rrangle$, and for every word $w$ in the symbols $\{e_0,e_1\}$, we have a function $f_w \in \cO(\pi_1^{\omega}(X,0))$ given by the coefficient extraction functional $f_w\colon \bQ \llangle e_0,e_1\rrangle \to \bQ$. Its restriction to primitive elements defines a linear map $\Lie(\pi_1^{\omega}(X,0)) \to \bQ$ which factors through the graded piece $\Lie(\pi_1^{\omega}(X,0))_{-n}$, where $n$ is the length of~$w$. For suitable combinations of fundamental group quotients~$\Pi$ and words~$w$, the function $f_w$ remains well-defined on the quotient $\Lie(\Pi^{\omega})_{-n}$ of $\Lie(\pi_1^{\omega}(X,0))_{-n}$.

\begin{prop}
	\label{polylog-fw-basis}
	 For $\Pi = \Pi_{\PL}$ the polylogarithmic quotient, a basis of $\Lie(\Pi_{\PL}^{\omega})_{-n}^{\vee}$ is given by $f_{e_0}$ and $f_{e_1}$ if $n = 1$, and by $f_{e_0^{n-1}e_1}$ if $n > 1$.
\end{prop}

\begin{proof}
	This is clear for $n=1$, so assume $n \geq 2$. Note that $f_{e_0^{n-1}e_1}$ is well-defined on the polylogarithmic Lie algebra $\Lie(\Pi_{\PL}^{\omega})$ since it vanishes on Lie monomials of $e_1$-degree $> 1$. Observe that the coefficient of the word $e_0^{n-1} e_1$ after expanding the Lie brackets in $\ad(e_0)^{n-1}e_1$ is $1$, i.e. we have 
	\[ f_{e_0^{n-1}e_1}(\ad(e_0)^{n-1}e_1) = 1. \]
	So $f_{e_0^{n-1}e_1}$ in fact agrees with the dual basis of the (singleton) basis $\ad(e_0)^{n-1}e_1$ of $\Lie(\Pi_{\PL}^{\omega})_{-n}$.
\end{proof}

To obtain a similar statement for the full fundamental group, we need the following lemma:

\begin{lemma}
	\label{Lyndon lemma}
	For every Lyndon word $w$ on $\{e_0, e_1\}$ we have
	\[ \lambda(w) = w + \sum_{u > w} a_u u \]
	in $\bQ\llangle e_0, e_1 \rrangle$, for certain $a_u \in \bZ$. In other words, $w$ is the lexicographically smallest monomial appearing in $\lambda(w)$.
\end{lemma}

\begin{proof}
	We prove this by induction. The case where $w$ is a single symbol is trivial. Assume $|w| \geq 2$ and let $w = w_1 w_2$ be the standard factorisation. By induction, we have
	\begin{align*}
		\lambda(w_1) &= w_1 + \sum_{u_1 > w_1} a_{u_1} u_1,\\
		\lambda(w_2) &= w_2 + \sum_{u_2 > w_2} b_{u_2} u_2,
	\end{align*}
	for certain $a_{u_1}, b_{u_2} \in \bZ$. This implies
	\begin{align*}
		\lambda(w) &= [\lambda(w_1), \lambda(w_2)] = \lambda(w_1)\lambda(w_2) - \lambda(w_2) \lambda(w_1)\\
		&= w_1 w_2 - w_2 w_1 + \text{linear combination of $w_1u_2, \; u_1 w_2,\; u_1 u_2,\; u_2 w_1,\; w_2 u_1,\; u_2 u_1$},
	\end{align*}
	with $u_i$ running through the words of length $|w_i|$ which are lexicographically larger than $w_i$. For all $u_1 > w_1$ and $u_2 > w_2$ we have
	\begin{align*}
		w_1 u_2 &> w_1 w_2,\\
		u_1 w_2 &> w_1 w_2,\\
		u_1 u_2 &> w_1 w_2,\\
		u_2 w_1 &> w_2 w_1,\\
		w_2 u_1 &> w_2 w_1,\\
		u_2 u_1 &> w_2 w_1.
	\end{align*}
	Since $w = w_1 w_2$ is a Lyndon word and $w_2 w_1$ is a cyclic permutation of it, we have $w_2 w_1 > w_1 w_2$. Thus, every monomial appearing in $\lambda(w)$ and different from $w$ itself is lexicographically larger than~$w$.
\end{proof}

\begin{prop}
	\label{lyndon-fw-basis}
	The functions $f_w$ with $w$ running through the Lyndon words of length~$n$ on $\{e_0,e_1\}$ form a basis of $\Lie(\pi_1^\omega(X,0))_{-n}^\vee$.
\end{prop}

\begin{proof} 
	Denote by $(D_w)_{w \text{ Lyndon}}$ the dual basis of the Lyndon basis of $\Lie(\pi_1^{\omega}(X,0))_{-n}$, i.e.\
	\[ D_w(\lambda(w')) = \delta_{w,w'} \]
	for any two Lyndon words $w$ and $w'$ of length~$n$. Since every element $\alpha$ of $\Lie(\pi_1^{\omega}(X,0))_{-n}$ can be written as $\alpha = \sum_{w\text{ Lyndon}} D_w(\alpha) \lambda(w)$, the $f_w$ and the $D_w$ are related by
	\[ f_u = \sum_{w \text{ Lyndon}} f_u(\lambda(w)) D_w \]
	for any Lyndon word~$u$. By \Cref{Lyndon lemma}, we have $f_w(\lambda(w)) = 1$ and $f_u(\lambda(w)) = 0$ whenever $u < w$, hence
	\[ f_u = D_u + \sum_{w < u \text{ Lyndon}} f_u(\lambda(w)) D_w. \]
	Thus, if the Lyndon words are indexed in increasing lexicographic order, the $f_w$ are obtained from the $D_w$ by applying a lower-triangular matrix with ones on the diagonal. This is invertible, hence the $f_w$ also form a basis.
\end{proof}

\begin{example}
	\label{Lyndon basis example}
	The Lyndon words of $(e_0, e_1)$-bidegree $(3,2)$ are $w_1 = e_0 e_0 e_0 e_1 e_1$ and $w_2 = e_0 e_0 e_1 e_0 e_1$. We have
	\begin{align*}
		\lambda(e_0 e_0 e_0 e_1 e_1) &= [e_0, [e_0, [[e_0, e_1], e_1]]] = e_0 e_0 e_0 e_1 e_1 - 2 e_0 e_0 e_1 e_0 e_1 + \ldots,\\
		\lambda(e_0 e_0 e_1 e_0 e_1) &= [[e_0, [e_0, e_1]], [e_0, e_1]] = e_0 e_0 e_1 e_0 e_1 + \ldots,
	\end{align*}
	where the ellipses represent a $\bZ$-linear combination of non-Lyndon words. Hence the $f_{w_i}$ and $D_{w_i}$ are related via the unipotent matrix
	\[ \left( \begin{matrix}
		f_{w_1}\\ f_{w_2}
	\end{matrix}\right) = \left( \begin{matrix}
		1 & 0 \\
		-2 & 1
	\end{matrix}\right) \left(\begin{matrix}
		D_{w_1}\\ D_{w_2}
	\end{matrix}\right). \]
\end{example}

\subsection{Iterated integral functions on the fundamental group}
\label{iterated-integral-functions}

When we think of $f_w$ as an algebraic function on $\pi_1^{\omega}(X,0)$ (rather than a linear functional $\bQ\llangle e_0,e_1\rrangle \to \bQ$) we also denote it by $\Li_w$ and call it a multiple polylogarithm function. In other words, for any word $w$ on $\{e_0,e_1\}$, the function 
\[ \Li_w\colon \pi_1^{\omega}(X,0) \to \bA^1 \]
sends an element $\gamma \in \pi_1^{\omega}(X,0)(R)$ valued in some $\bQ$-algebra~$R$ to the $w$-coefficient of the expansion of $\gamma$ as a group-like power series in $\bQ\llangle e_0,e_1\rrangle$. Thus, tautologically, we have
\[ \gamma = \sum_w \Li_w(\gamma) w. \]
With the isomorphism $\pi_1^{\omega}(X,0) \otimes_{\bQ} K \cong \pi_1^{\dR}(X,0)$, the function $\Li_w$ is given by the Tannakian iterated integral over a sequence of the differential forms, with $e_0$ corresponding to $\omega_0 \coloneqq \mathrm{d}t/t$ and $e_1$ corresponding to $\omega_1 \coloneqq \mathrm{d}t/(1-t)$:
\[ \Li_w(\gamma) = \int_{\gamma} \omega_{i_1}\cdots \omega_{i_r} \quad \text{for $w = e_{i_1}\ldots e_{i_r}$, $i_j \in \{0,1\}$, $\gamma \in \pi_1^{\dR}(X,0)$}. \]
Formula~\eqref{eq:shuffle-product} then translates into the shuffle product formula for iterated integrals:
\[ \left(\int_{\gamma} \omega_{i_1}\cdots \omega_{i_r}\right) \left(\int_{\gamma} \omega_{i_{r+1}}\cdots \omega_{i_{r+s}}\right) = \sum_{\pi \in \Sh(r,s)} \int_{\gamma} \omega_{i_{\pi(1)}}\cdots \omega_{i_{\pi(r+s)}}. \]

As discussed in the previous subsection, the coefficient extraction functionals $f_w$ also define linear functions on the Lie algebra $\Lie(\pi_1^{\omega}(X,0))$. They are related to the functions $\Li_w$ on $\pi_1^{\omega}(X,0)$ as follows:

\begin{lemma}
	\label{fw-and-Li-w}
	For $\gamma \in \pi_1^\omega(X,0)$ and any word $w$ on $\{e_0,e_1\}$ we have
	\[ f_w(\log(\gamma)) = \sum_{m=1}^\infty \frac{(-1)^{m+1}}{m} \sum_{\substack{w=w_1\cdots w_m\\ |w_i| \geq 1}} \Li_{w_1}(\gamma) \cdots \Li_{w_m}(\gamma).\]
	Here the inner sum is over all decompositions of $w$ into $m$ nonempty subwords.
\end{lemma}

Note that the sum is finite since $w$ can be split into at most $|w|$ nonempty parts.

\begin{proof}
	We view $\gamma$ as a group-like element of $R\llangle e_0,e_1 \rrangle$, for a $\bQ$-algebra~$R$, and use the logarithm series:
	\begin{align*}
		f_w(\log(\gamma)) &= \sum_{m=1}^\infty \frac{(-1)^{m+1}}{m} f_w((\gamma -1)^m)\\
		&= \sum_{m=1}^\infty \frac{(-1)^{m+1}}{m} \sum_{w=w_1\cdots w_m} f_{w_1}(\gamma-1)\cdots f_{w_m}(\gamma-1).
	\end{align*}
	Since $\gamma$ has constant term~1, we have
	\[ f_{w_i}(\gamma - 1) = \begin{cases}
			\Li_{w_i}(\gamma), & \text{if $|w_i| \geq 1$,}\\
			0, & \text{if $w_i$ is empty}.
	\end{cases} \]
	Thus, we arrive at the claimed formula.
\end{proof}

\begin{prop}
	\label{thm: coordinates on full group}
	\leavevmode
	\begin{enumerate}[label=(\alph*)]
		\item The ring of functions on $\pi_1^\omega(X,0)$ is the polynomial algebra in the functions $\Li_w$ with $w$ running through the Lyndon words.
		
		\item The ring of functions on $\Lie(\pi_1^\omega(X,0))$ is the polynomial algebra in the functions $f_w$ with $w$ running through the Lyndon words.
	\end{enumerate}
\end{prop}

\begin{proof}
	The second statement is a direct consequence of \Cref{lyndon-fw-basis}. The first one follows from the second via pullback along the logarithm isomorphism $\pi_1^\omega(X,0) \cong \Lie(\pi_1^\omega(X,0))$ since $\log^{\sharp}(f_w)$ equals $\Li_w$ up to a polynomial in $\Li_u$'s of smaller degree by \Cref{fw-and-Li-w}. Alternatively, the first statement follows from Radford's Theorem on polynomial generators of shuffle algebras.
\end{proof}

\subsection{Polylogarithmic functions}

Consider the polylogarithmic words 
\[ e_0 \quad \text{and} \quad e_0^{n-1} e_1 \text{ for $n \geq 1$}. \]
We use special notation for their coefficient extraction functions:

\begin{defn}
	\label{def: polylog functions}
	On $\Lie(\pi_1^\omega(X,0))$ define
	\[ L_0 \coloneqq f_{e_0}, \quad L_n \coloneqq f_{e_0^{n-1} e_1} \text{ for $n \geq 1$}. \]
	On $\pi_1^\omega(X,0)$ define
	\[ \Log \coloneqq \Li_{e_0}, \quad \Li_n \coloneqq \Li_{e_0^{n-1} e_1} \text{ for $n \geq 1$}. \]
\end{defn}

Note that the polylogarithmic words are Lyndon words, so in particular $\Log$ and $\Li_n$ for $n \geq 1$ are algebraically independent. The subalgebra of $\cO(\pi_1^\omega(X,0))$ generated by them is closed under the coproduct, hence is a sub-Hopf algebra. This is the ring of functions on the polylogarithmic quotient $\pi_1^{\omega}(X,0)_{\PL}$.

As a special case of \Cref{fw-and-Li-w}, we can express the pullback of the functions $L_n$ along the logarithm isomorphism in terms of polylogarithms. Applying the lemma to $w = e_0$ we obtain
\begin{equation}
	\label{eq: L_0}
	L_0(\log(\gamma)) = \Log(\gamma).	
\end{equation}

For $w = e_0^{n-1}e_1$ we obtain the following:

\begin{prop}
	\label{thm: L_n in terms of polylogs}
	For $\gamma \in \pi_1^\omega(X,0)$ and $n \geq 1$ we have
	\begin{equation}
		\label{eq: L_n}
		L_n(\log(\gamma)) = \sum_{k=0}^{n-1} \frac{B_k}{k!} \Log(\gamma)^k \Li_{n-k}(\gamma).
	\end{equation}
	Here $B_k$ denotes the $k$-th Bernoulli number (with the convention $B_1 = -1/2$).
\end{prop}

\begin{proof}
	According to \Cref{fw-and-Li-w}, we have
	\begin{align*}
		L_n(\log(\gamma)) &= \sum_{m=1}^\infty \frac{(-1)^{m+1}}{m} \sum_{\substack{e_0^{n-1}e_1=w_1\cdots w_m\\ |w_i| \geq 1}} \Li_{w_1}(\gamma) \cdots \Li_{w_m}(\gamma),
		\intertext{with the inner sum running over all decompositions of $e_0^{n-1}e_1$ into $m$ nonempty subwords $w_1,\ldots,w_m$. Let $\ell_i \geq 1$ be the length of $w_i$, then the subwords are given by $w_i = e_0^{\ell_i}$ for $1 \leq i \leq m-1$ and $w_m = e_0^{\ell_m-1}e_1$. So we write}
		\ldots &= \sum_{m=1}^\infty \frac{(-1)^{m+1}}{m} \sum_{\substack{\ell_1+\ldots+\ell_m = n\\ \ell_i \geq 1}} \Li_{e_0^{\ell_1}}(\gamma) \cdots \Li_{e_0^{\ell_{m-1}}}(\gamma) \Li_{e_0^{\ell_m-1}e_1}(\gamma).
		\intertext{The shuffle relation \Cref{lem: group-like and primitive coefficients}\,\ref{item: group-like coefficients} implies $\Li_{e_0^\ell}(\gamma) = \frac1{\ell!}\Log(\gamma)^{\ell}$,}
		\ldots &= \sum_{m=1}^\infty \frac{(-1)^{m+1}}{m} \sum_{\substack{\ell_1+\ldots+\ell_m = n\\ \ell_i \geq 1}} \frac{1}{\ell_1!}\cdots \frac1{\ell_{m-1}!} \Log(\gamma)^{\ell_1+\ldots+\ell_{m-1}} \Li_{\ell_m}(\gamma),
		\intertext{setting $k = \ell_1 + \ldots + \ell_{m-1}$:}
		\ldots &= \sum_{m=1}^\infty \frac{(-1)^{m+1}}{m} \sum_{k=m-1}^{n-1} \left( \sum_{\substack{\ell_1+\ldots+\ell_{m-1} = k\\ \ell_i \geq 1}}  \frac{1}{\ell_1!}\cdots \frac1{\ell_{m-1}!} \right)  \Log(\gamma)^k \Li_{n-k}(\gamma),
		\intertext{switching order of summation and shifting $m$ by $1$:}
		\ldots &= \sum_{k=0}^{n-1} \underbrace{\left( \sum_{m=0}^k  \frac{(-1)^m}{m+1} \sum_{\substack{\ell_1+\ldots+\ell_m = k\\ \ell_i \geq 1}}  \frac{1}{\ell_1!}\cdots \frac1{\ell_m!} \right)}_{\eqqcolon B(k) \in \bQ} \Log(\gamma)^k \Li_{n-k}(\gamma).
	\end{align*}
	We wish to determine the rational numbers $B(k)$ appearing as coefficients.
	For any tuple $\underline{\ell} = (\ell_1,\ldots,\ell_m)$ with $\ell_i \geq 1$ and $\sum \ell_i = k$, define a sequence $\underline{a} = (a_1,a_2,\ldots)$ where $a_i$ is the number of $i$'s appearing in $\underline{\ell}$. Thus, $a_i \geq 0$, $\sum a_i = m$, $\sum i a_i = k$ and $\prod \frac1{\ell_i!} = \prod \left(\frac1{i!}\right)^{a_i}$. For given $\underline{a}$, the number of ways to arrange the $a_i$ copies of $i$ in an $\underline{\ell}$-tuple for all $i$ is given by the multinomial coefficient $\binom{a_1+a_2+\ldots}{a_1, a_2,\ldots} = (\sum a_i)!/\prod a_i!$. So we can write
	\[ B(k) = \sum_{m=0}^k \frac{(-1)^m}{m+1} \sum_{\substack{a_1,a_2,\ldots \geq 0\\ \sum a_i = m\\ \sum ia_i = k}} \binom{m}{a_1,a_2,\ldots} \prod_{i=1}^\infty \left(\frac1{i!}\right)^{a_i}. \]
	
	We compute the generating function
	\begin{align*}
		\sum_{k=0}^\infty B(k) t^k &= \sum_{k=0}^\infty \sum_{m=0}^k \frac{(-1)^m}{m+1} \sum_{\substack{a_1,a_2,\ldots \geq 0\\ \sum a_i = m\\ \sum ia_i = k}} \binom{m}{a_1,a_2,\ldots} t^k \prod_{i=1}^\infty \left(\frac1{i!}\right)^{a_i}.\\
		\intertext{Using $t^k = t^{\sum ia_i} = \prod (t^i)^{a_i}$, we can eliminate the summation over $k$,}
		\ldots &= \sum_{m=0}^\infty \frac{(-1)^m}{m+1} \sum_{\substack{a_1,a_2,\ldots \geq 0\\ \sum a_i = m}} \binom{m}{a_1,a_2,\ldots} \prod_{i=1}^\infty \left(\frac{t^i}{i!}\right)^{a_i},
		\intertext{and use the multinomial theorem:}
		\ldots &= \sum_{m=0}^\infty \frac{(-1)^m}{m+1} \left( \sum_{i=1}^\infty \frac{t^i}{i!} \right)^m\\
		&= \sum_{m=0}^\infty \frac{(-1)^m}{m+1} (e^t - 1)^m\\
		&= \frac1{e^t-1} \sum_{m=0}^\infty \frac{(-1)^m}{m+1} (e^t - 1)^{m+1}\\
		&= \frac1{e^t-1} \log(e^t)\\
		&= \frac{t}{e^t-1}.
	\end{align*}
	We recognise this as the exponential generating function of the Bernoulli numbers, hence $B(k) = B_k/k!$.
\end{proof}

\section{Coordinates on the Selmer scheme and the motivic Kummer map}

We apply the construction of Selmer scheme coordinates from §\ref{sec:coordinates-on-selmer-scheme} to the case where $\Pi$ is the depth $N$ quotient of the full fundamental group or of the polylogarithmic fundamental group. Recall that for every $\sigma \in \Lie(U_S^{\MT})_{-n}$ and $f \in \Lie(\Pi^{\omega})_{-n}^{\vee}$ we have a Selmer function
\[ \Phi^{\sigma}_f\colon \Sel^{\mot}_{S,\Pi}(X) \to \bA^1 \]
defined by $\Phi^{\sigma}_f(c) = f(c(\sigma))$ for a graded cocycle $c\colon \Lie(U_S^{\MT}) \to \Lie(\Pi)$ (\Cref{def:selmer-scheme-functions}). When $f = f_w$ is a coefficient extraction functional for a word $w$ on $\{e_0,e_1\}$, we simplify the notation by writing $\Phi_w^{\sigma} \coloneqq \Phi^{\sigma}_{f_w}$.

\begin{thm}
	\label{selmer-coordinates-full-and-polylog}
	Let $\Sigma = \coprod_{n=1}^\infty \Sigma_{n}$ be a free homogeneous generating system of $U_S^{\MT}$.
	\leavevmode
	\begin{enumerate}[label=(\alph*)]
		\item For the polylogarithmic depth~$N$ Selmer scheme we have an isomorphism
		\[ \Sel_{S,\PL,N}^{\mot}(X) \cong \Spec \bQ[(\Phi^{\sigma}_w)_{n,\sigma,w}] \]
		with $1 \leq n \leq N$, $\sigma \in \Sigma_n$, and $w$ ranging over the polylogarithmic words of length~$n$ ($e_0$ and $e_1$ for $n = 1$; $e_0^{n-1}e_1$ for $n \geq 2$).
		\item For the full depth $N$ Selmer scheme we have an isomorphism
		\[ \Sel_{S,N}^{\mot}(X) \cong \Spec \bQ[(\Phi^{\sigma}_w)_{n,\sigma,w}] \]
		with $1 \leq n \leq N$, $\sigma \in \Sigma_n$, and $w$ ranging over the Lyndon words of length~$n$ on $\{e_0,e_1\}$.
	\end{enumerate}
\end{thm}

\begin{proof}
	This is \Cref{coordinates-on-selmer-scheme} applied with the bases of $\Lie(\Pi^{\omega})_{-n}^{\vee}$ from Propositions~\ref{polylog-fw-basis} and~\ref{lyndon-fw-basis}.
\end{proof}

\subsection{Constructing generators of \texorpdfstring{$U_S^{\MT}$}{U\_S\textasciicircum{}\{MT\}}}
\label{sec:generators-of-lie-algebra}

So far we have not discussed the construction of the Lie algebra generators $\Sigma = \coprod_{n=1}^{\infty} \Sigma_n$ of $U_S^{\MT}$. We give a brief summary of \cite[§3.2]{DC:mixedtate2}, which addresses this question. Set $Z = \Spec(\cO_{K,S})$ and write
\[ A(Z) \coloneqq \cO(U_S^{\MT}) \]
for the ring of (unipotent mixed Tate) motivic periods over~$Z$. We have $A(Z) = \bigoplus_{n \geq 0} A(Z)_n$ via the grading by half-weight. The subspace of primitive elements inside $A(Z)_n$ is denoted by $E_n(Z)$. There is a canonical isomorphism
\[ E_n(Z) \cong \Ext^1_{\MT(\cO_{K,S},\bQ)}(\bQ(0),\bQ(n)) \]
and $E_n(Z)$ is called the subspace of \emph{extensions}. 
We also have the subspace $D_n(Z) \subseteq A(Z)_n$ of \emph{decomposables}, defined as the image of the multiplication map 
\[ \bigoplus_{\substack{i+j=n\\ i,j \geq 1}} A(Z)_i \otimes A(Z)_j \to A(Z)_n. \]
The subspaces $D_n(Z)$ and $E_n(Z)$ are linearly disjoint, thus we have an injective linear map
\[ E_n(Z) \hookrightarrow A(Z)_n/D_n(Z). \]
Its linear dual is canonically isomorphic to the abelianisation map
\[ \Lie(U_S^{\MT})_{-n} \twoheadrightarrow \Lie(U_S^{\MT})_{-n}^{\ab}. \]
Thus, an element~$\overline{\sigma}$ of $\Lie(U_S^{\MT})_{-n}^{\ab}$ can be defined by specifying its values on a basis of $E_n(Z)$. A lift $\sigma$ of $\overline{\sigma}$ along the abelianisation map is uniquely determined by requiring that $\sigma$ vanishes on a given complementary subspace $P_n$ of $E_n(Z) \oplus D_n(Z) \subseteq A(Z)_n$.

Elements of $A(Z)_n$ can be constructed via motivic iterated integrals: for any word $w$ in $\{e_0,e_1\}$ of length~$n$ and any $S$-integral point $\alpha \in X(\cO_{K,S})$, we have an element $\Li_w^{\fu}(\alpha) \in A(Z)_n$. By allowing $\alpha$ to be a tangential base point, we also have the motivic multiple zeta values $\zeta^{\fu}(w) \coloneqq \Li^{\fu}_w(-\vec{1}_1) \in A(Z)_n$. Constructing bases of the spaces $A(Z)_n$ and their subspaces $E_n(Z)$ and $D_n(Z)$ and a complementary subspace~$P_n$ is a difficult problem. That this can be achieved via iterated integrals on $\bP^1 \smallsetminus \{0,1,\infty\}$ is the content of a conjecture by Goncharov. See \cite[§4.3]{CDC:polylog1} for a calculation where such bases are constructed in the case $K = \bQ$, $S = \{3\}$, $n=1,2,3$.

For $n = 1$, a basis of $E_1(Z) = A(Z)_1$ can always be constructed via motivic logarithms. For $\alpha \in \cO_{K,S}^{\times}$ we have $\log^{\fu}(\alpha) \in A(Z)_1$, defined as a motivic iterated integral on~$\bG_m$. The map
\[ \log^{\fu}\colon \cO_{K,S}^{\times} \to A(Z)_1 \]
satisfies $\log^{\fu}(\alpha\beta) = \log^{\fu}(\alpha) + \log^{\fu}(\beta)$, hence it extends to a $\bQ$-linear map on $\cO_{K,S}^{\times} \otimes \bQ$.

\begin{lemma}
	\label{log-basis}
	Let $\alpha_1,\ldots,\alpha_r$ be a basis of $\cO_{K,S}^{\times} \otimes \bQ$. Then the elements $\log^{\fu}(\alpha_1),\ldots,\log^{\fu}(\alpha_r)$ form a basis of $E_1(Z) = A(Z)_1$ and we can choose the Lie algebra generators $\Sigma_1 = \{ \tau_1,\ldots,\tau_r\}$ in half-weight $-1$ as the dual basis. \qedhere
\end{lemma}

\subsection{The motivic Kummer map}

If $\alpha \in X(\cO_{K,S})$ is an $S$-integral point of~$X$, we have a \emph{motivic path torsor} $\pi_1^{\mot}(X;0,\alpha)$ \cite{deligne-goncharov}. This is a $\pi_1^{\mot}(X,0)$-torsor in $\MT(\cO_{K,S},\bQ)$. For any quotient $\pi_1^{\mot}(X,0) \twoheadrightarrow \Pi$, define $\pathtorsor{\alpha}{0}$ as the pushout of $\pi_1^{\mot}(X;0,\alpha)$ along the quotient map. Then $\pathtorsor{\alpha}{0}$ is a $\Pi$-torsor in $\MT(\cO_{K,S},\bQ)$, hence it defines a $\bQ$-rational point of the Selmer scheme $\Sel_{S,\Pi}^{\mot}(X)$. The \emph{motivic Kummer map} is defined as
\begin{align*}
	j\colon X(\cO_{K,S}) &\to \Sel_{S,\Pi}^{\mot}(X)(\bQ),\\
	\alpha &\mapsto \left[\pathtorsor{\alpha}{0}\right].
\end{align*}
The same construction works when $\alpha$ is an $S$-integral tangent vector.

Let $\Pi = \pi_1^{\mot}(X,0)$ be the full fundamental group. The Selmer coordinates $\Phi_w^{\sigma}$ of a Kummer element in $\Sel_{S,\Pi}^{\mot}(X)$ are given as follows:

\begin{prop}
	\label{kummer-element-coordinates}
	Let $\sigma \in \Lie(U_S^{\MT})_{-n}$ and let $w$ be a word of length~$n$ on $\{e_0,e_1\}$. Then for $\alpha \in X(\cO_{K,S})$ we have
	\[ \Phi^{\sigma}_w(j(\alpha)) = \langle \sigma, \Li_w^{\fu}(\alpha) \rangle. \]
\end{prop}

\begin{proof}
	Let $c_{\alpha}\colon U_S^{\MT} \to \Pi^{\omega}$ be the $\bG_m$-equivariant cocycle representing the path torsor $j(\alpha) = [\pi_1^{\mot}(X;0,\alpha)]$. The motivic polylogarithm $\Li_w^{\fu}(\alpha)$ is given by $c_{\alpha}^{\sharp}(\Li_w)$. Write $c_{\alpha}$ also for the extension to a Hopf algebra cocycle $\cH(U_S^{\MT}) \to \cH(\pi_1^{\mot}(X,0))$. By definition of the Selmer function $\Phi^{\sigma}_w$, we have
	\[ \Phi^{\sigma}_w(j(\alpha)) = f_w(c_{\alpha}(\sigma)) = (c_{\alpha}^{\sharp}f_w)(\sigma) = \langle \sigma, \Li_w^{\fu}(\alpha) \rangle. \qedhere \]
\end{proof}

\Cref{kummer-element-coordinates} gives an idea of why it can be advantageous to work on the level of Hopf algebras: the definition of the Selmer scheme coordinates uses linear functions on the Lie algebra $\Lie(\Pi^{\omega})$, whereas the construction of the Lie algebra generators $\Sigma = \coprod_{n=1}^{\infty} \Sigma_n$ of $U_S^{\MT}$ involves specifying their values on motivic periods, which are functions on the group $U_S^{\MT}$. The complete Hopf algebras $\cH(U_S^{\MT})$ and $\cH(\Pi^{\omega})$ provide an ambient structure for both the groups and their Lie algebras, giving the flexibility to freely switch between the two.

\printbibliography

\end{document}